\documentclass[11pt,a4paper]{article}
\usepackage{epsfig}

\usepackage{hyperref}
\usepackage{amsmath,amsthm,amssymb,latexsym}
\usepackage{amsfonts}
\usepackage[american]{babel}
\usepackage{mathrsfs}

\newtheorem{theorem}{Theorem}[section]
\newtheorem{definition}[theorem]{Definition}
\newtheorem{lemma}[theorem]{Lemma}
\newtheorem{cor}[theorem]{Corollary}
\newtheorem{rem}[theorem]{Remark}
\newtheorem{proposition}[theorem]{Proposition}
\newtheorem{ex}[theorem]{Example}

\newcommand{\relint}{{\mathrm{relint}}\,}
\newcommand{\cl}{{\mathrm{cl}}\,}

\newcommand{\diam}{{\mathrm{diam}}\,}

\newcommand{\vare}{\varepsilon}

\newcommand{\R}{{\Bbb R}}

\def\j1n{j=1,\dots,n}
\def\j1m{j=1,\dots,m}
\def\i1np1{\in +1}

\def\R{\mathbb{R}}

\def\rn{\mathbb{R}^n}

\def\i1np1{\in +1}

\def\R{\mathbb{R}}

\def\rn{\mathbb{R}^n}

\def\u1{u^{(1)}}

\def\h1{h^{(1)}}

\newcommand{\la}{\lambda}

\newcommand{\ga} {\gamma}
\newcommand{\Ga} {\Gamma}

\newcommand{\om}{\omega}
\newcommand{\Om} {\Omega}

\newcommand{\RR}{{\mathbb R}}

\newcommand{\NN}{{\mathbb N}}

\newcommand{\pa}{\partial}

\newcommand{\lims} {\lim\limits}
\newcommand{\dist} {\mbox{\rm dist }}

\newcommand{\mis}{\mbox{\rm mis }}
\begin{document}

\begin{title}{On   steepest descent curves for quasi convex families in $\RR^n$}
\end{title}
%\Subjclass{35J25, 35J65} \Keywords{Elliptic equations, convexity
%of level sets, quasi-concave envelope, Minkowski addition}

\author{Marco Longinetti\footnote{marco.longinetti@unifi.it, Dipartimento DiMaI,
Universit\`a degli Studi di Firenze, V.le Morgagni 67,  50134
Firenze - Italy}\\
Paolo Manselli\footnote{paolo.manselli@unifi.it, %%
Firenze - Italy}\\
Adriana Venturi\footnote{adriana.venturi@unifi.it, Dipartimento
GESAAF, Universit\`a degli Studi di Firenze, P.le delle Cascine 15,
50144 Firenze - Italy}}
\date{}
\maketitle

\begin{small}{\bf Abstract. }A connected, linearly ordered path $ \ga\subset \rn $
satisfying
$$x_1, x_2,
x_3 \in \ga, \; \mbox{ and } \; x_1 \prec x_2 \prec x_3 \quad
\Longrightarrow \quad  |x_2-x_1| \leq |x_3-x_1|
$$
 is shown to be a rectifiable curve; a priori bounds for its length are
 given; moreover, these paths are generalized  steepest descent curves 
of  suitable quasi convex functions. Properties of quasi convex families  are  considered; special curves
related to  quasi convex families are defined and studied; they
are generalizations of steepest descent curves for quasi convex functions and satisfy the
previous property. Existence, uniqueness, stability results and
length's bounds are proved for them.

\bigskip

{\bf R\'esum\'e.} Nous d\'emontrons que les chemins  $ \ga\subset \rn $ qui sont connect\'es et ordonn\'es, avec la  propriet\'e de monotonicit\'e $$x_1, x_2,
x_3 \in \ga, \; \mbox{ et } \; x_1 \prec x_2 \prec x_3 \quad
\Longrightarrow \quad  |x_2-x_1| \leq |x_3-x_1|
$$ sont des courbes. Des limitations pour leur longueur sont prouv\'e.
Ces chemins  sont des  g\'en\'eralisations de courbes de la plus grande pente pour appropri\'ees fonctions quasiconvexes.
Propri\'et\'es des familles quasiconvexes et courbes li\'ees avec elles sont \'etudi\'ees.  Nous d\'emontrons l'existence,
l'unicit\'e, la d\'ependance continue de ces courbes avec des limitations pour leur longueur.   \end{small}
\newline
\smallskip

2000 \emph{Mathematics Subject Classifications. Primary }52A20; \emph{ Secondary} 52A10, 52A38, 49J53.

\emph{Key words and phrases.} Quasiconvex functions, steepest descent curves.

\section{Introduction}
Given a smooth, real valued function in a domain $ A \subset \rn$,
$f : A \to \RR $, with gradient $ Df \ne 0, $  steepest descent
curves are solutions to
\begin{equation}\label{gradientequation}
\dot{x}(t)=\phi(Df(x(t))) Df(x(t)), \quad \phi > 0
\end{equation}
(sometimes the condition $ \phi < 0 $ is preferred). The steepest
descent curves satisfy the geometrical fact to be
orthogonal to the level sets $ \{x:f(x)= \mbox{const}\} $. This
fact shows that the trajectories of the  steepest descent curves depend on
the level sets of $ f $ only.

The definition of steepest descent curves has been extended to more
general functions and spaces in \cite{Brezis}, \cite{Degio-Marino},
\cite{Degio-Marino2}, \cite{Degio-Marino3}, \cite{Bolte},
  \cite{Marcellin}.

Here generalized steepest descent curves will be defined and studied for quasi convex
  functions without smoothness assumptions; let us refer to \cite{Crouzeix} for properties of these functions.
In particular they will be studied for lower semi continuous  bounded functions $f$
  with compact convex sub level sets:
$$
\Om_\tau=\{x :f(x)\leq \tau \} \quad \inf f  \leq \tau \leq
\sup f .
$$
%In this work, the geometrical characterization of steepest descent
%curves, noted above for smooth functions, will be studied, suitably
%modified, for quasi convex functions.

Let $ \{ \Om_{\tau} \} $ be a nested family of bounded compact
convex sets, sub level sets of a bounded, lower semi continuous,
quasi convex function; this family, according to \cite{fenc}, will
be called a \textbf{ quasi convex family}. Let us assume here, for
simplicity, that  $ \mathit{Int}( \Om_{\tau}) $, the set of its interior
points, is not empty. In the work it can happen that  $Int(Q)=\emptyset$ for some $Q$ in the family.  
The condition for a curve $ x(\cdot)$ to be a steepest
descent curve for a quasi convex family will be that, 
\begin{equation}\label{steepdesceqnconvfam}
  x(\tau) \in
\partial \Om_{\tau} \quad \forall  \tau\quad \mbox{and} \quad\dot{x}(\tau) \in N_{\Om_{\tau}} \quad \mbox{a.e.} ,\end{equation}
 where  $ N_{\Om_{\tau}} $ is the
normal cone to $ \Om_{\tau} $ at $ x(\tau);$ i.e. $x(\cdot)$ is a time-dependent  trajectory of   a
differential inclusion, see e.g.  \cite[\S 4.4]{Aubin}.
Properties of steepest descent curves for quasi convex families were observed in
\cite{Manselli-Pucci}: rectifiables curves were studied that are
steepest descent curves for some  quasi convex family $\{\Om_{\tau}\}$ and
 bounds were proved for their length. It was noticed  that, for these curves,
the following property holds:
\begin{equation}\label{selfexp}
\mbox{if}\quad  t_1 < t_2 < t_3 \quad \mbox{then} \quad |x(t_1) -
x(t_2)| \leq |x(t_1) - x(t_3)|
\end{equation}
(with the opposite orientation a continuous curve satisfying  (\ref{selfexp}) was called
self contracting curve, see \cite{Daniilidis2}.
The authors \cite{Daniilidis} and \cite{Daniilidis2} studied them
  and proved, for $
n=2 $ bounds for their length).
If a parametrization $x(\cdot)$  for the curve $\ga$ is available, the following  extension of both  above conditions \eqref{steepdesceqnconvfam},\eqref{selfexp} will be
used in this work:   
\begin{equation}\label{exco}
\forall \Omega_{\tau},\; \forall y \in  \Omega_{\tau},\; \text{if}\quad x(\tau)
\notin Int (\Omega_{\tau}), \, \text{then}\quad\forall t> \tau: \quad
|x(\tau)-y| \leq |x(t)-y|.
\end{equation}
Sometimes, in the paper,  in place of the quasi convex family $ \{ \Om_{\tau} \} $, will be also considered a convex
stratification $\mathfrak{F}$ (definition \ref{defstratifications}). 
The class of convex stratifications  (see \cite{Defi}) is  more general than the class of quasi convex families. 
A convex stratification is not necessarily parameterized by a continuous parameter. 
Special curves $\ga$ associated to a convex stratification $\mathfrak{F}$ satisfying definition 
\ref{defEC+} (a generalization of (\ref{exco})) are considered and studied. 
The couple $(\ga,\mathfrak{F})$ is called \textbf{ Expanding
Couple} (EC).
In the present work, two  problems are addressed and solved.

First, properties of connected paths $\ga$ included in a convex body $\Om \subset \rn$
 and satisfying (\ref{selfexp})
only (that will be called \textbf{Self Expanding Paths} (SEP))  are
studied: the SEP turn out to be  rectifiable curves with a priori
bounded length, depending only on the dimension $n$ and on the mean
width of $\Om$ (theorem \ref{selfsecrettifiable}).

Second, properties of   convex stratifications $\mathfrak{F}$ and expanding couples  are  considered. For any expanding couple  $(\ga,\mathfrak{F})$ regularity properties of the 
 curve $\ga$, associated to  $\mathfrak{F}$,
are studied. Existence, uniqueness  and stability  results,
   (theorems \ref{coruniqcondEC},  \ref{uniqcondEC}) are proved for $(\ga,\mathfrak{F})$.
Moreover $\ga$ can be parameterized in a such way that its representation $x(\cdot)$ is a lipschitz continuous time-dependent  trajectory of  a differential inclusion of the type
(\ref{steepdesceqnconvfam}) (theorem \ref{abscont}). Also it is obtained that, if $(\ga,\mathfrak{F})$ is  an EC and  $ \Omega_{1} \subset \Omega_{2}\in  \mathfrak{F}$  then, the length of the part of $ \ga $ between $
\Omega_{1} $ and  $ \Omega_{2} $ satisfies the bound:
$$
\mbox{length}\, \big( \ga \cap ( \Omega_{2} \setminus
\Omega_{1} ) \big) \; \leq \; \mbox{const} \cdot \mbox{dist}
(\Omega_{1} , \Omega_{2} )
$$
with  Hausdorff distance and   constant depending  only on the dimension $n$ (see theorem \ref{uovo}).

 Results on the maximal length of a steepest descent curve for convex functions 
 were noticed as an important tool for studying them (see \cite{Bolte}); 
the previous inequality provides an apriori bound for their lengths.

 The main tools in our
approach are: first, the suitable parametrization of the self
 expanding paths with respect to the mean width of the convex hulls of the increasing
 parts of the curve; second,  the parametrization of the quasi convex families with respect
  to their  mean width.
The structure of the present work follows. In \S2 and \S3
preliminary facts are stated and properties on cap bodies and the
variation of their mean width are studied. In \S4 SEP are introduced
and regularity results for them are proved. The main result in theorem \ref{selfsecrettifiable} is that
a SEP,  with the mean width  parametrization, is Lipschitz continuous. It is shown that its length can
be bounded a priori. In \S5 quasi convex stratifications and quasi convex families with  their
properties are considered. In \S6 steepest descent curves and expanding couples  are defined and studied; existence and
uniqueness problems are stated and solved. Our approach about existence and regularity properties of steepest descent curves do not require 
that the time-dependent trajectory $x(\cdot)$ in \eqref{steepdesceqnconvfam} is absolutely continuous.
%In example \ref{exmio} it is shown indeed that this can occur, even if $t$ is the mean width %arametrization of the quasi convex family.

\section{Preliminaries and definitions}
Let
$$
B(z,\rho)=\{x\in\rn\,:\,|x-z|<\rho\}\,,\quad\,
S^{n-1}=\partial B(0,1) \, \quad n\geq 2.
$$
%If $X\in\rn\setminus\{0\}$, we set
%$$\theta=\frac{X}{|X|}\in S^{n-1}\,.$$

A nonempty,  compact convex set $K$ of $\rn$ will be called  a {\em
convex body}. $Int(K)$ and $\pa K$ denote the interior of $K$ and
the boundary of $K$, $cl(K)$  is    the closure of $K$,
$\mathit{Aff}(K)$ will be the smallest affine hull containing $K$,
and $\relint K$, $\pa_{rel} K$ are  the corresponding subsets in
the topology of $\mathit{Aff}(K)$.
 For any set $S$, $co(S)$ is the convex hull of $S$. $Lin^+ (S)$ is the smallest linear
  space containing $S\cup\{0\}$, $Lin^- (S)$ is the largest linear  space contained in
   $S\cup\{0\}$. $Lin^+, Lin^-$ operate in the vector space structure of  $\rn$.
For  a convex body $K\subset \rn$, the {\em support function} is
defined by

\begin{equation*}
\label{equation1} H_K(x)=\sup_{y\in K}\langle x,y\rangle\,,\quad
x\in\mathbb{R}^n,
\end{equation*}
where $\langle \cdot, \cdot \rangle$ denotes the  scalar
product in $\mathbb{R}^n$. The restriction of $H_K$ to $S^{n-1}$ will be  denoted by
$h_K$.

It is well known that, if $0\le \lambda \le 1$, $A$ and $B$ are
convex bodies,  so is $\lambda A+(1-\lambda)B$  and
\begin{equation}\label{sommaH}
H_{\lambda A+(1-\lambda)B}=\lambda H_A+(1-\lambda) H_B\,;
\end{equation}
let us  recall that $H$ is monotone with respect to inclusion, i.e.
\begin{equation}\label{monotoneH}
A\subseteq B\quad\text{if and only if}\quad H_A(x)\leq
H_B(x)\quad\forall x\in\rn\,.
\end{equation}
 The width of a convex set $K$  in a
direction $\theta \in S^{n-1}$ is the distance between the two
hyperplanes  orthogonal to $\theta$ and  supporting $K$,  given by
$h_K(\theta)+ h_K(-\theta)$. The {\em mean width } $w(K)$ of $K$ is
the mean of this distance on $S^{n-1}$ with respect to the spherical
Lebesgue measure $\sigma$, i.e.
\begin{equation}\label{meanwidthdef}
w(K)= \dfrac{1}{\omega_n}\int_{S^{n-1}} \left( h_K(\theta)+ h_K(-\theta)\right)\,
d\sigma= \dfrac{2}{\omega_n} \int_{S^{n-1}}  h_K(\theta)\, d\sigma
\end{equation}
where $\omega_n=2\pi^{n/2}/\Gamma(n/2)$ is the  measure  of   $S^{n-1}$.

\begin{rem}
If  $K$ is a convex body in $\rn$  and $k=\dim \mathit{Aff}(K)< n,
$ there are competing mean widths for $K$: $w_k(K)$ the mean width
of $K$ as subset of  $\mathit{Aff}(K)$, $w_n(K)$ the mean width of
 $K$ as subset of  $\rn$. Let us recall that
  $w_k(K)/w_n(K)=\frac{\omega_{k+1}}{\omega_{k}}\frac{\omega_n}{\omega_{n+1}}$ is a constant
  depending on $n$ and $k$ only.
\end{rem}
  In what follows $w(K)$ will always be $w_n(K)$.

Let $K$ be a convex body and $q\in  K$; the {\em normal cone } at
$q$ to $K$ is the closed convex cone

\begin{equation}\label{normalcone} N_K(q)=\{x\in\rn: \langle x,y-q\rangle \le 0 \quad
\forall y \in K\}.
\end{equation}
 When $q \in Int(K)$ then  $N_K(q)$ reduces to zero.

\begin{definition}
Let $K$ be a convex body  and $p$ be a point not in K. A  simple
cap body $K^p$ is:
\begin{equation}\label{capbody}
K^p=\bigcup_{0\le \lambda \le 1}\{ \la K+(1-\la)p\}=co(K\cup\{p\}).
\end{equation}
\end{definition}

Cap bodies properties can be found in \cite{Bonnfen}. For later use let us define also $K^p=K$ for $p\in K$.

\begin{proposition} \label{propsupportcapbody}Let  $K$  be  a convex body,
 $N_p=N_{K^p}(p)$ the normal cone to $K^p$ at $p$, then
\begin{equation}\label{supportcapbody}
H_{K^p}(x) = \left \{ \begin{array}{lc}
                     H_K(x)  &   for \quad x\not \in N_p, \\
                    \langle x, p \rangle & for \quad x \in N_p.
                   \end{array}
                     \right.
\end{equation}
\end{proposition}
\begin{proof}
Let $x \in N_p $ , by  (\ref{normalcone})
$$
\langle x, z \rangle \le \langle x, p \,\rangle \quad \forall z\in
K^p,$$
 so $H_{K^p}(x)\le \langle x, p \,\rangle .$ Since $K^p\supset
P=\{p\}$, monotone property \eqref{monotoneH} implies
\begin{equation}\label{HKp>HK}
H_{K^p}(x)\ge H_P(x)= \langle x, p \,\rangle \quad \forall x\in \rn,
\end{equation}
then, last part of \eqref{supportcapbody} holds. If $ x\not \in
N_p$, by (\ref{normalcone}) there exist  $\overline{z}\in K$ and $
\overline{\la}\in (0,1]$, satisfying
$$
\langle x, [\overline{\la}\overline{z}+(1-\overline{\la})p]-p
\rangle
>0,
$$
i.e. $\langle x ,\overline{z}-p\rangle >0$. Then
$$
 H_K(x) > \langle x, p \rangle.
$$
Let $z_1\in K, \la_1\in [0,1]$ satisfying: $H_{K^p}(x)=\langle x,
\la_1z_1+(1-\la_1)p \rangle.$ Then
$$
H_{K^p}(x)\leq \max \{ \langle x,z_1 \rangle , \langle p, x \rangle
\} \leq H_K(x).
$$
As $K^p \supset K$ , $H_{K^p}(x) \geq H_K(x)$ and the first equality
in \eqref{supportcapbody} holds.\end{proof}
Let us notice that in previous proposition we dont assume that $p\not \in  K$,
 but later it will be used  in that case.

The {\em dual cone} $C^*$ of a convex cone $C$ is
$$C^*=\{y\in \rn : \langle y,x \rangle \ge 0 \quad \forall x \in C\}.$$
The dual cone $C^*$ is a closed and convex cone.

In \cite[theorem 5]{fenc}, it is observed that for a convex  cone $C$
$$\dim Lin^+ (C) + \dim Lin^- (C)= n. $$
The opening of a circular cone will be  the amplitude of the acute angle between the
axis and a generator half line. If $C$ is a circular cone of opening $\alpha$ then $C^*$
is a circular cone of opening $\pi/2-\alpha$.
The  {\em tangent cone}, or  support cone, of a convex body  K at a
point $q \in \pa K$  is given by
$$ T_K(q)=cl \{\bigcup_{y\in K} \{s( y-q): s \geq 0  \}\}.$$
It is well known that:
\begin{equation}\label{N=-T*}
N_K(q)=-T^*_K(q).
\end{equation}
Moreover:
\begin{enumerate}
\item[(i)] $\dim \mathit{Aff}(T_K(q))=\dim \mathit{Aff}(K)$;
\item[(ii)] $\dim Lin^-(N_K(q))=n-\dim \mathit{Aff}(K)$;
\item[(iii)] if $p\in \mathit{Aff} (K)$ then $\dim \mathit{Aff}(K^p)= \dim \mathit{Aff}(K)$;\\
 if $p\not\in \mathit{Aff} (K)$ then $\dim \mathit{Aff}(K^p)= \dim \mathit{Aff}(K)+1$.
\end{enumerate}

\subsection{Sequences of cap bodies and their normal cones}\label{simplecapbody}

\begin{proposition}\label{limvarepsilonNp} For $ \varepsilon >0$ let $p_{\varepsilon}\not \in K$ and 
let $\lim_{\varepsilon \to 0^+}p_{\varepsilon}=p_0\in \pa K$. 
Then
\begin{equation}\label{limsupNp}
\limsup_{\varepsilon \to 0^+} N_{K^{p_\varepsilon}}(p_\varepsilon)\subseteq N_{K}(p_0).
\end{equation}
\end{proposition}
\begin{proof}
The  proposition is a consequence of standard results in the theory of convex analysis,
however the proof is elementary, arguing by converging sequences. Let $\varepsilon_r \to 0^+$
for $r \to \infty$ and let for simplicity  $p_r=p_{\varepsilon_r}$, $N_r=N_{K^{p_{r}}}(p_r)$,
for $r \in \NN$. Let $\{x_r\}_{r \in \NN}$ be a converging sequence of points in $N_r$ and
 $x_0=\lims_{r\to \infty} x_r$. By construction
\begin{equation}\label{ineqpr}
 \langle x_r,y \rangle \le \langle x_r, p_r \rangle \quad \forall y \in K^{p_r},
 \forall r \in \NN.
\end{equation}
Since $K\subset K^{p_r}$, the previous inequality holds for all $y \in K$; going to the
 limit, one gets
$$ \langle  x_0 ,y\rangle \le \langle x_0, p_0 \rangle \quad \forall y \in K.$$
This proves that $x_0\in N_{K}(p_0)$.
\end{proof}
Let $u\neq 0, u\in \rn$. Let us consider  the half space
$\{u\}^*=\{x\in\rn: \langle  x, u \rangle \ge 0\},$
 i.e. the  dual cone of the half line starting from the origin through $u$.
\begin{proposition}\label{inclusionNpvarepsilon}
Let  $\varepsilon > 0 $, $K$ be a convex body and $p_0\in \pa K$. Let $u\not \in T_K(p_0)$
 and let  $p_{\varepsilon}=p_0+\varepsilon u$.  Then
$$ \{u\}^*\supset N_{K^{p_\varepsilon}}(p_\varepsilon)\supseteq \, \, N_{K}(p_0)\cap \{u\}^*.$$
\end{proposition}
\begin{proof}As $u\not \in T_K(p_0)$, then $p_\varepsilon\not \in K$  and  $K^{p_\varepsilon}$
is a simple cap body.
Let $x \in N_{K^{p_\vare}}(p_\vare)$, this means
$$ \langle x, \la q+(1-\la)p_\vare -p_{\vare} \rangle \leq 0, \quad \forall q\in K,
\la \in [0,1],$$
i.e.
\begin{equation}\label{dpropu*}
 \langle x, \la (q-p_{0}-\vare u )\rangle \leq 0.
\end{equation}
Then $ \langle x, u \rangle \geq 0$ so $x \in \{u\}^*$. If  $x \in
N_{K}(p_0)$, then $ \langle x, q-p_0 \rangle \leq 0  \quad\forall q
\in K$. If also  $x \in \{u\}^*,$ thus for every $\la \geq 0,$
\eqref{dpropu*} is satisfied and
 $ x \in N_{K^{p_\vare}}(p_\vare)$.
\end{proof}

\begin{theorem}\label{limvarepsilonNpu} Let $K$ be a convex body and $p_0\in \pa K$. Let
 $u\not\in T_K(p_0), u\neq 0$ and let  $p_{\varepsilon}=p_0+\varepsilon u$.  Then
 $p_{\varepsilon}\not \in  K$ and
\begin{equation}\label{limsupNpu}
\lim_{\varepsilon \to 0^+}N_{K^{p_\varepsilon}}(p_\varepsilon)= N_{K}(p_0)\cap \{u\}^*.
\end{equation}
\end{theorem}
\begin{proof}

Propositions \ref{limvarepsilonNp}, \ref{inclusionNpvarepsilon} imply
\begin{equation}\label{limsupsubset}
\limsup_{\varepsilon \to 0^+} N_{K^{p_\varepsilon}}(p_\varepsilon)\subseteq N_{K}(p_0)
\cap \{u\}^*.
\end{equation}
 Proposition \ref{inclusionNpvarepsilon} implies also
$\liminf_{\varepsilon \to 0^+} N_{K^{p_\varepsilon}}(p_\varepsilon)\supseteq \, N_{K}(p_0)
\cap \{u\}^*$.
 \end{proof}

\section{First variation of the mean width of cap bodies}\label{meanwidth}
Let us recall that the Hausdorff distance  between two convex bodies $A$ and $B$ can be
written as
$$dist(A,B)= \max_{\theta \in S^{n-1}}|h_A(\theta)-h_B(\theta)|$$
(see \cite[theorem 1.8.11]{Schn}).

Let $K$ be a convex body, $p_0\in \pa K$. Let us  consider  the
variation of $K$  by deforming $K$ as a simple cap body in a given
direction $u \not\in T_K(p_0)$. More precisely let   $p_{\varepsilon
}=p_0+\varepsilon u$ as in proposition \ref{inclusionNpvarepsilon}
and $N(p_0)=N_{K^{p_0}}(p_0)$, $
N(p_\varepsilon)=N_{K^{p_\varepsilon}}(p_\varepsilon)$. When $N$ is
a cone let  $\widehat{N}$ be  the {\em sector} $N\cap S^{n-1}$.
\begin{theorem}\label{firstvariationwidth} Let $\varepsilon   > 0$,  $p_0\in \pa K$,
 $u\not\in T_K(p_0), u\neq 0$, $p_{\varepsilon }=p_0+\varepsilon u$, then
\begin{equation}\label{firstvariationwidtheq}
\frac{\omega_n}2\left( w(K^{p_\varepsilon})-w(K)\right)=\varepsilon
\int_{ \widehat{N(p_0)}\cap \{u\}^*}\langle  \theta , u \rangle \,\,
d\sigma(\theta) + \int_{\widehat{N(p_\varepsilon)}\setminus
(\widehat{N(p_0)} \cap\{u\}^*)}\left( h_{K^{p_\varepsilon}}(\theta
)-h_K(\theta )\right)d\sigma(\theta).
\end{equation}
The last integral of the right hand side
 is  positive and infinitesimum of order greater than one  for $\varepsilon \to 0^+$.
\end{theorem}
\begin{proof} Let use formula \eqref{meanwidthdef} for the mean width.
The integral  of $(h_{K^{p_\varepsilon}}(\theta)-h_{K}(\theta))$ on $S^{n-1}$  can be
 split on three sets
$$\widehat{N(p_0)}\cap \{u\}^*, \quad \widehat{N(p_\varepsilon)}\setminus
(\widehat{N(p_0)} \cap \{u\}^*), \quad (S^{n-1}\setminus \widehat{N(p_\varepsilon)})
\cap \{u\}^*.$$
Since $K\subset K^{p_\varepsilon}$,
by proposition \ref{propsupportcapbody} we have
$$ h_{K^{p_\varepsilon}}(\theta)-h_K(\theta)=0 \text{\quad \quad for \quad}
\theta \not \in \widehat{N(p_\varepsilon)}  .$$
Moreover for $\theta \in \widehat{N(p_0)}\cap\{u\}^*$ (which is included in
$\widehat{N(p_\varepsilon)}$ by  proposition \ref{inclusionNpvarepsilon}), we have
$$h_{K^{p_\varepsilon}}(\theta)=h_{\{p_\varepsilon\}}(\theta)= \langle \theta ,
 p_\varepsilon\rangle,$$
$$h_{K}(\theta)=h_{\{p_0\}}(\theta)= \langle \theta , p_0\rangle.$$
Therefore
$$ h_{K^{p_\varepsilon}}(\theta)-h_{K}(\theta)=\varepsilon \langle \theta , u\rangle
\text{\quad \quad for \quad} \theta \in \widehat{N(p_0)}\cap\{u\}^* ,$$
 and formula \eqref{firstvariationwidtheq} is proved.

Since the Hausdorff distance between $K^{p_\varepsilon}$ and
$K$ is less  than $\varepsilon |u|$, then for any  $\theta$:
$$| h_{K^{p_\varepsilon}}(\theta)-h_K(\theta) | \leq \vare |u|. \quad $$
   Theorem \ref{limvarepsilonNpu} implies that
$$\mis \left(\widehat{N(p_\varepsilon)}\setminus (\widehat{N(p_0)}\cap\{u\}^*)
\right) \to 0  \text{\quad \quad for \quad}
\varepsilon \to 0^+.$$
This proves that the last integral in \eqref{firstvariationwidtheq} is
infinitesimum of order greater than one of $\varepsilon$ for $\varepsilon \to 0^+$
 and it is positive since  the cap body $K^{p_\varepsilon}$ contains $K$.
\end{proof}

The differential properties of $w(K^p)$  has  been investigated in $\RR^3$ in
 \cite[Satz VI]{Stoll}.
\begin{proposition}
$w(K^p)$ is a convex function of $p$ and for $p\not  \in \pa K$ is differentiable  with
$$\nabla w(K^p)=\frac{2}{\omega_n}\int_{\widehat{N(p)}}\theta\, d\sigma .$$
\end{proposition}
\begin{proof} Let $p,q\in \rn$;
 Let $p_{\la}=\la p+(1-\la)q$, $0\leq \la \leq 1$, then $K^{p_\la}\subseteq
 \la K^{p}+(1-\la)K^q$, and
$h_{K^{p_\la}} \leq \la h_{K^p} +(1-\la)h_{K^{q}}$, therefore
$$w(K^{p_\la})\leq  \la w({K^p}) +(1-\la)w({K^{q}}).$$
Hence $w(K^p)$ as a function of $p$ is convex. If $p\in Int(K)$ then $N(p)=\{0\}$ and the
thesis follows trivially. When $p\not \in K$, from equality \eqref{firstvariationwidtheq}
(with $K^p$ in place of $K$, $p$ in place of $p_0$, with $n$ choices of vectors $u$
linearly independents and not in $T_{K^p}(p)$) the thesis follows.
\end{proof}

Next proposition is a consequence of corollary 2 in \cite{Vitale}.
Let us give a shorter proof in our simpler situation of nested
convex bodies.
\begin{proposition}\label{problemII} For any  two convex bodies $\Om_1\subset \Om_2$
of $\RR^n$, the following inequality holds
\begin{equation}\label{al=1/2}
\sqrt[n]{\frac{c^{(0)}_n}{(\diam(\Om_2))^{n-1}}}\dist(\Om_2,\Om_1)
%\sqrt[n]{\frac{c^{(0)}_n}{\Big(\sqrt{(\frac{\dist(\Om_2,\Om_1)}{2}
%)^2+(\diam(\Om_1)^2}\;\Big)^{n-1}}}\dist(\Om_2,\Om_1)
\leq (w(\Om_2)-w(\Om_1)^{1/n},
\end{equation}
 where $c^{(0)}_n$ depends only on
$n$.
\end{proposition}
\begin{proof}Let $\dist(\Om_2,\Om_1) >0$.
Let $p\in \pa \Om_2,\, q \in \pa \Om_1$ such that
$|p- q|=\dist(\Om_2,\Om_1),$ and let $p_0$ be the mid point on the segment $pq$; since
$$\Om_2\supset\Om_1^p \supset\Om_1^{p_0}\supset\Om_1  $$ then
\begin{equation}\label{meanwidthringppo}
 w(\Om_1^p)-w(\Om_1^{p_0})  \leq w(\Om_2)-w(\Om_1).
\end{equation}

Let $u$ be the unit vector $(p-p_0)/|p-p_0|$ and let $N_{p_0}$ be the normal cone at
 $p_0$ to $\Om_1^{p_0}$. First let us observe that $  N_{p_0}\subseteq \{u\}^*$. From
  (\ref{firstvariationwidtheq})
$$w(\Om_1^p)-w(\Om_1^{p_0})\geq \frac{2}{\omega_n}|p-p_0|\int_{ \widehat{N(p_0)}}\langle
 \theta , u \rangle \,\, d\sigma(\theta).$$
The tangent cone at $p_0$ at $\Om_1^{p_0}$ is contained in the circular cone $C_{u,\beta}$
 with axis in direction of $u$ and opening $\beta=\arctan (\frac{\diam(\Om_1)}{|q-p_0|})$.
 Therefore
$N(p_0)\supseteq -C_{u,\beta}^*=C_{u,\pi/2-\beta},$
and from lemma \ref{lemmappendix0}
$$\int_{ \widehat{N(p_0)}}\langle  \theta , u \rangle \,\, d\sigma(\theta)\geq
\int_{ \widehat{C_{u,\pi/2-\beta}}}\langle  \theta , u \rangle \,\,
d\sigma(\theta)= \frac{\om_{n-1}}{n-1}(\cos \beta)^{n-1}.$$ Since
$2|q-p_0|=2|p-p_0|=\dist(\Om_2,\Om_1)$ and $\tan
\beta=\frac{\diam(\Om_1)}{|p-p_0|}$,
 it follows that
$$\cos^2\beta=(1+\tan \beta^2)^{-1}=\dfrac{|p-p_0|^2}{|p-p_0|^2+\diam^2(\Om_1)}
\geq \dfrac{|p-p_0|^2}{2(\diam^2(\Om_2))};$$
 (\ref{al=1/2}) follows  from  the previous three inequalities with
  $c^{(0)}_n=2^{-(n-1)}\frac{\om_{n-1}}{(n-1)\om_n}$.
\end{proof}
On the other hand from \eqref{meanwidthdef}

\begin{equation}\label{distw<distH}
w(\Om_2)-w(\Om_1) \leq \frac{2}{\omega_n}\max_{S^{n-1}}|h_2(\theta)-h_1(\theta)|=
\frac{2}{\omega_n}\dist(\Om_2,\Om_1).
\end{equation}

\section{ Self expanding paths}\label{selfexpanding}
\begin{definition}\label{defselfexpandinglinearly}
Let us call  self expanding path (SEP)  a non empty, closed,
connected  and linearly  strictly ordered
 (by $\prec$, with $x_1\prec x_2\Rightarrow x_1\neq x_2$) subset $\ga$  of $\rn$,
 with the property:
\begin{equation}\label{IBDCmonlinearly}
x_1, x_2,
x_3 \in \ga, \; \mbox{ and } \; x_1 \prec x_2 \prec x_3 \quad
\Longrightarrow \quad  |x_2-x_1| \leq |x_3-x_1|.
\end{equation}
\end{definition}
This class of paths,   was studied as class of curves in
\cite{Manselli-Pucci} with an added rectifiability hy\-po\-thesis;
differential properties and bounds were obtained. Here
 properties are proved from the above geometric definition with  no rectifiability
 assumptions. 
\begin{rem}
Let us notice that the graph  of a continuous and monotone real
 function $f$ on a bounded interval is a SEP with ordering
  $p_1\equiv(t_1,f(t_1)) \prec p_2\equiv (t_2,f(t_2))$ iff $t_1 < t_2$.
\end{rem}

\begin{definition}\label{defgax}
Let $x_0\in \ga$, let us denote
$$\ga_{x_0}=\{x\in \ga : x\prec x_0\}\cup \{x_0\}.$$
 \end{definition}
\begin{proposition}\label{proprsdc} If $\ga$ is a self expanding path and $x\in\ga$,
  then for any $p,q \in co(\ga_x)\setminus \{x\}$
\begin{equation}\label{angcond}
\langle p-x,q-x\rangle >0,
\end{equation}
and any two half lines from $x\in \ga$ in the tangent cone at $co(\ga_x)$ are the
sides of an  angle  less than or equal to $\pi/2$.
\end{proposition}
\begin{proof} It is enough to prove inequality \eqref{angcond} for $ x_1,x_2 \in \ga$,
 $x_1\prec x_2 \prec  x$. From \eqref{IBDCmonlinearly}
$$0<|x_2-x_1|\leq |x-x_1|;$$
therefore  the triangle  of vertices $x,x_1,x_2$ has an acute angle at the vertex $x$.
\end{proof}
\begin{cor}
At  any point $p$ of any  self expanding path  $\ga$ the  inclusion
\begin{equation}\label{NpsubsetTp}
N_{co(\ga_p)}(p)\supseteq -T_{co(\ga_p)}(p)
\end{equation}
holds.
\end{cor}

From now on it will be assumed  that all the self expanding paths $\ga$  considered are
contained in a closed ball $\tilde{B}$, and let $\Ga$ this class. Of course if $\ga \in\Ga $
and $x_0\in \ga$, then $\ga_{x_0}\in \Ga$. The path $\ga$ with the topology induced by $\rn$
is a metric space. Let us notice  also that $Int (co(\ga_x))$ can be an  empty set, i.e $\dim \mathit{Aff}(co(\ga_x))) <n.$

\begin{lemma}\label{geomSEP}
Let  $\ga \in \Ga$. The following properties hold
\begin{enumerate}
\item[(i)] if $x\in \ga , x\neq \min \ga $, then $x \in \pa_{rel} co(\ga_x)$;\\
\item[(ii)] $x_1 \prec x_2 \in \ga   \Rightarrow x_2 \not\in co(\ga_{x_1})$;\\
\item[(iii)]  $x_1 \prec x_2 \in \ga   \Rightarrow w(co( \ga_{x_1})) < w(co( \ga_{x_2}))$.
\end{enumerate}
\end{lemma}

\begin{proof}
(i) follows from proposition \ref{proprsdc}: if $x$ is not on the relative boundary of
 $co(\ga_x)$ then the  tangent cone at $x$ will be $\mathit{Aff}(co(\ga_x))$, in
 contradiction with \eqref{angcond}.
If  (ii) does not hold,  then $x_2 \in co( \ga_{x_1})$. On the other hand $x_2$ has
positive distance from the compact set $\ga_{x_1}$; then $x_2$ must be  in the interior of
a segment with end points $y,z\in co(\ga_{x_1})\subset co(\ga_{x_2 })$, in contradiction
with \eqref{angcond}. (iii) follows from (ii) since $co( \ga_{x_1}) \subset
\neq co( \ga_{x_2})$ and the mean width is a strictly increasing function in the class of
 convex bodies with respect the inclusion relation.
\end{proof}

\begin{definition}
Let us call a parametrization of a path $\ga\in\Ga$ a mapping of a real interval
 $T$, $\ T\ni t\to x(t) \in \ga$, satisfying $\forall x_0\in\ga$, $\{t\in T: x(t)=x_0 \}$ is
 an interval (possibly reduced to a point), and
$$x_0,x_1 \in \ga,  x_0\prec x_1  \Longrightarrow \sup\{t\in T:x(t)=x_0 \}
 <\inf\{t\in T:x(t)=x_1 \}.$$
A parametrization will be called continuous if $T $ is a closed
interval and  $x(\cdot) $is continuous in  $T.  $
\end{definition}
\begin{theorem}\label{geomSEPiv}Let  $\ga \in \Ga$. The path
$\ga$ has a one-to-one continuous parametrization $x(w)$,  inverse of
\begin{equation}\label{pargaw}
w(x):=w(co(\ga_x))\in [0, w(co(\ga))].
\end{equation}
\end{theorem}

\begin{proof}
From (iii) of previous lemma the map $w: \ga \ni x \to w\in [0, w(co(\ga))]$ is injective.
Moreover if $x_1\prec x_2$, $|x_2-x_1| <\vare $ then the ball $B(x_1,\vare)$ contains all
 the points of $\ga$ between $x_1$ and $x_2$.
Therefore $$co( \ga_{x_2})\subset co( \ga_{x_1})+B(x_1,\vare).$$
Hence $w(co( \ga_{x_2}))-w( co( \ga_{x_1})) \leq 2\vare$. Then, if
$\ga$ is equipped with the topology induced by $\rn$, the map $w$ is
continuous  and maps $\ga$ in a connected  subset of $[0,
w(co(\ga))]$. Since $\ga$ is compact $w$  has minimum and maximum;
by (iii) they are $0,$ and $w(co(\ga))$ respectively. Thus $w$ is
bijective and continuous; its inverse:
$$[0, w(co(\ga))]\ni w \to x(w)$$
is a continuous parametrization of $\ga$.
\end{proof}

Let $\ga$ be a self expanding path  with a continuous
parametrization $x(\cdot)$ defined in a real interval $T$. Let
$\ga(t)=\ga_{x(t)}$. Let us notice that  the integer valued function
$ T \ni t \to \dim \mathit{Aff}(co(\ga(t)))$ is not decreasing and
left continuous.

The following property comes out from  elementary
  geometry and will be used later.
\begin{lemma}\label{lemmaecissec}  Let  $p,q,y_i\in \rn,  \,  i=1, \ldots,s  $. If
\begin{equation}\label{axesineq}
|p-y|^2 \leq |q-y|^2, \quad \mbox{for \quad } y=y_i, \,  i=1, \ldots,s
\end{equation}
then the same holds for any  $y\in co(\{y_i, i=1,\ldots, s\})$.
\end{lemma}
\begin{proof} It is enough to prove that if \eqref{axesineq} holds for $y_1,y_2$
then, it holds with  any $y$ on the line segment $y_1y_2$. Now
\eqref{axesineq} is equivalent to claim  that both  $y_i,\,i=1,2$
belong to the
 half space containing $p$ and delimited by the hyperplane orthogonal to the line segment
  $pq$ at the middle point of it. Since such half space is convex, it contains all the points $y$
  on  $y_1y_2$. Therefore the  inequality \eqref{axesineq} holds for all $y\in y_1y_2$.
\end{proof}

$||\ga||=\int_T |\dot{x}(t)| dt$ denotes the   length  of a rectifiable curve $\ga$,
with a parametrization  $T \ni t\to x(t)$.
\begin{theorem}\label{selfsecrettifiable} Let $\ga\in \Gamma$ be a  self expanding
path in $\rn$. Then   $\ga, $ parameterized by the mean width
function \eqref{pargaw} is Lipschitz continuous, and a.e.
\begin{equation}\label{dx/dw<=c(n)}
 |\dfrac{dx}{dw}| \leq c^{(1)}_n,
\end{equation}
where $ c^{(1)}_n $ is  a constant  depending only on the dimension
$n$. In particular  any
 self expanding path $\ga$ (which is connected) is a rectifiable curve and
\begin{equation}\label{||ga||<=c(n)}
||\ga||\leq  c^{(1)}_n \cdot w(co(\ga)).
\end{equation}
\end{theorem}

\begin{proof}

{\em Step a)}: Let $x(w), w\in [0, w(co(\ga))]$ be the continuous parametrization of $\ga$
introduced in theorem
\ref{geomSEPiv}. Let $0=w_0 < \cdots < w_i <\cdots w_s=w(co(\ga))$ the decomposition of
$[0, w(co(\ga))]$ such that if $w\in (w_i, w_{i+1}]$, $i=0,..,s-1$ then,
 $\dim \mathit{Aff}(co(\ga(w)))$ has constant value $m_i$.
It is sufficient to  prove that $x(w)$ is lipschitz continuous in
$[w_i, w_{i+1}]$ and \eqref{dx/dw<=c(n)} holds. If $m_0=1$, $[w_0,
w_1]\ni w \to x(w) $ is linear and its derivative  with respect to
the one dimensional  mean width is trivially $1$; then,
\eqref{dx/dw<=c(n)} holds in $[w_0, w_1]$ with $ c^{(1)}_n = \pi
\frac{\omega_n}{\omega_{n+1}}$.\\

{\em Step b)}: It can be assumed that $m_i \geq 2, i=0,..,s-1$.
For every  $\overline{w}\in ( w_i,w_{i+1})$, the relative interior of $co(\ga(\overline{w}))$
 is non empty; let $B(\tilde{y},\rho_1)$ a $m_i$-dimensional ball such that
$$B(\tilde{y},\rho_1)\subset \relint co(\ga(\overline{w})).$$
Let us fix $\overline{w} \in (w_i, w_{i+1})$. Then for every $w'\in
[\overline{w},w_{i+1}]$
$$B(\tilde{y},\rho_1)\subset \relint co(\ga(\overline{w}))\subset \relint co(\ga(w')).$$
In the remaining part of this   step and in the steps c) and d),  for simplicity, we will be
 arguing with $m_i=n.$

Let $p'=x(w')$, $K_{p',1}$ be the convex cone with opening $\alpha=\alpha(w')$ such that
$ p'+K_{p',1}$ is tangent to the ball $B(\tilde{y},\rho_1)$; from proposition \ref{proprsdc}
it follows that $ 0<\alpha \leq \pi/4$ and
$$0<\rho_1\leq \frac{\tilde{y}-p'}{\sqrt{2}}\leq \frac{\diam (co(\ga))}{\sqrt{2}}.$$
 Let  $K_{p',1/2}$ be the convex cone, with the same axis as $K_{p',1}$ and  opening
  $\alpha/2$.  Then
$$K_{p',\lambda}\subset T_{co(\ga(w'))}(p'), \quad \mbox{for \quad }
 \lambda=1,\frac12,\quad w'\in [\overline{w},w_{i+1}].$$

As a consequence of \eqref{NpsubsetTp}
\begin{equation}\label{-(Kp',1/2)^*}
-K_{p',1}\subset N_{co(\ga(w'))}(p')= -(T_{co(\ga(w'))}(p'))^*\subset -(K_{p',1/2})^*.
\end{equation}
Moreover if a unit vector $u\in -(K_{p',1/2})^*$ then 
\begin{equation}\label{u*}
\{u\}^*\supset -K_{p',1/2},
\end{equation} i.e.
\begin{equation}\label{uinK*}
\langle u, \theta \rangle \geq 0 \quad \mbox{if} \quad \theta\in -K_{p',1/2},
\end{equation}
and as consequence $u \not \in T_{co(\ga(w'))}(p')$.\\

{\em Step c)}: The aim of this step is to prove that there exists a
constant $C=C(\rho_1)$ so that, for every $w''\in [w',w_{i+1}]$,
\begin{equation}\label{P''-p'<C}
|x(w'')-x(w')| < C  \Longrightarrow x(w'')-x(w') \in -(K_{p',1/2})^*.
\end{equation}
Let   $y\in  \ga(\overline{w})$. As $\ga$ is a SEP, the real
function  $ w \to |x(w)-y| $ is  not decreasing for $w\geq
\overline{w}$ ; then from lemma \ref{lemmaecissec} the same holds  for every $y\in
co( \ga(\overline{w}))$; in particular
\begin{equation}\label{fence}
|x(w'')-y|\geq |x(w')-y|=|p'-y|, \quad \forall y \in B(\tilde{y},\rho_1),\, w''\geq w'
\geq \overline{w}.
\end{equation}
Let  $p''=x(w'').$ The inequalities \eqref{fence} imply:
 $$p''\not \in \Phi_{p'}:=\bigcup_{v\in \pa (p'+K_{p',1})\cap\pa
  B(\tilde{y},\rho_1)}  B(v, |p'-v|).$$
The boundary of  $p'- (K_{p',1/2})^*$ intersects the boundary of
 $\Phi_{p'}$ in a $(n-2)$-dimensional sphere whose points have distance
  from $p'$ given by $l=\rho_1\frac{\cos \alpha}{\cos \alpha/2}$, (see Figure 1).
\begin{figure}[htb]
\epsfig{file=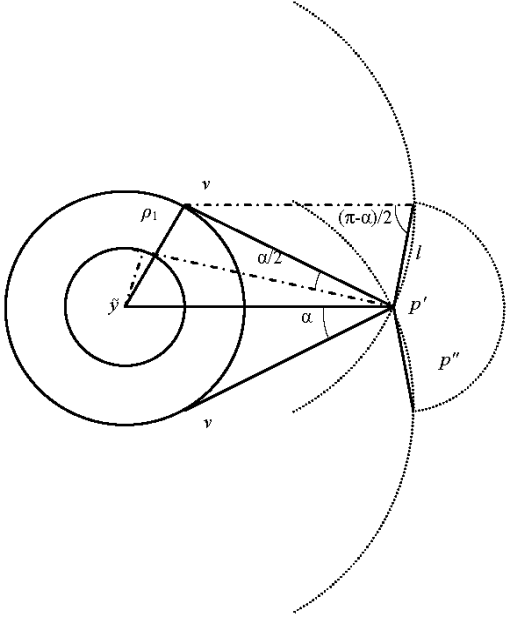, width=10cm} \caption{The boundary of $
\Phi_{p'}$.}
\end{figure}
If $p''$ has distance from $p'$ less than $l$, then
 $p''\in p'- (K_{p',1/2})^*$, i.e.
$$|p''-p'| <\rho_1\frac{\cos \alpha}{\cos \alpha/2}\Longrightarrow p''- p' \in
-(K_{p',1/2})^*.$$ Since  $\phi(\alpha)=\frac{\cos \alpha}{\cos
\alpha/2} $ is decreasing in $[0, \pi/4],$ then \eqref{P''-p'<C} is
proved with
$$ C=\rho_1\cdot \min \phi = \rho_1\frac{\cos \pi/4}{\cos \pi/8  }.$$

{\em Step d)}: The goal of this step is to  prove that $x(w)$ is locally lipschitz
continuous in the open   interval $(w_i, w_{i+1})$ with $n\geq 2$. Let $\delta > 0$;
 uniform continuity of  $x(\cdot)$ in $[\overline{w},w_{i+1}]$ guarantees that there
 exists $h_0$ such that $0< h< h_0$ implies that $|x(w'+h)-x(w')| < \delta.$ Let us choose
  $w''=w'+h$, $\delta=C(\rho_1)$, then $u$, the unit direction of $x(w'+h)-x(w')$, belongs
  to $-(K_{p',1/2})^*$.
As $|p''-p'| < \delta $, from the previous step,  \eqref{uinK*} holds
and
 $u\not \in T_{co(\ga(w'))}(p')$;  thus $\{u\}^*\cap N_{co(\ga(w'))}(p')\supseteq
  -(K_{p',1/2})$. Let us notice that (using the cap body notation)
$$(co(\ga(w')))^{p''}\subset co(\ga(w'')).$$
Since  \eqref{uinK*} holds, theorem  \ref{firstvariationwidth} can be applied with $K= co(\ga(w'))$, $p_0=p'$ ,
  $p_\vare=p''$ and $  u  $ as above.  It follows that
\begin{equation}\label{firstvariationwidthineq}
(w'+h)-w'\geq w(K^{p''})-w(K)\geq  |p''-p'|\frac{2}{\omega_n}\int_{\widehat{N(p')}
\cap\{u\}^* } \langle \theta,u\rangle d \sigma (\theta),
\end{equation}
 with $N_{co(\ga(w'))}(p')=N(p')$.
From \eqref{-(Kp',1/2)^*} we have $N_{p'} \supseteq -K_{p',1/2}$, and with \eqref{u*} we get
$$\int_{\widehat{N(p')}\cap\{u\}^* } \langle
\theta,u\rangle d \sigma \geq \int_{\widehat{-K_{p',1/2}}} \langle
\theta,u\rangle d \sigma.$$ From lemma \ref{lemmappendix1} in  the
appendix, last integral is bounded from below from a positive
constant $C(n, \alpha)=\frac{\omega_{n-1}}{n-1}
\cdot\sin^n(\frac{\alpha}{4}) $; since $C(n, \alpha)$ is increasing
for
 $0 < \alpha \leq \pi/4$ and  $\alpha$  is greater than
  $\overline{\alpha}=\arcsin (\rho_1/ \diam(co(\ga)))$, then
  $C(n, \alpha)>C(n,\overline{\alpha})$, and this bound is uniform in
  $[\overline{w},w_{i+1}]$.
This fact proves that $x(w)$ is lipschitz continuous in
$[\overline{w}, w_{i+1}]$  with a constant depending on  $\overline{w}, \diam (\ga)$.\\

{\em Step e)}: As $\overline{w}$ is arbitrary in $(w_i,w_{i+1})$,
$x(w)$ is locally lipschitz in that interval, thus rectifiable. By
using   \cite[theorem VII]{Manselli-Pucci} in
 every closed subinterval of $(w_i,w_{i+1})$,
\begin{equation}\label{lipschitxcmi}
|x(w'')-x(w')| \leq c(m_i)|w''-w'|,
\end{equation}
holds, where $ c(m_i)$ depends on $m_i$ only. As a consequence
\eqref{lipschitxcmi} holds in $   [w_i,w_{i+1}]$ and \eqref{dx/dw<=c(n)}
follows.\end{proof}

In \cite{Manselli-Pucci} it has been proved that $  c^{(1)}_2 = \pi
$ (best possible constant) and
 $$
  c^{(1)}_n \leq (n-1) \cdot n^{n/2} \frac{\omega_n}{\omega_{n-1}} .
$$
If we drop in the definition \ref{defselfexpandinglinearly} the assumption that $\ga$ is
connected, then the continuity of the inverse of the map $w$ in theorem \ref{geomSEPiv}
 does not hold. As example the piecewise steepest descent curves considered in
 \cite[definition 15]{Bolte} are not connected, moreover it was proved that they  are
 not rectifiable. Here  only connected curves are considered; of course, if $\ga$ is not
 connected, the corresponding properties  hold  for each connected component of $\ga$.

Let us notice that there exist   SEP with  not absolutely continuous parametrization.
 To provide an example let us consider the Cantor function $[0,1] \ni t \to g(t)$,
  see \cite[p. 83]{Halmos}, which is a not  decreasing function, with zero derivative a.e. in
  $[0,1]$, not absolutely continuous, with $g(0)=0, g(1)=1$.
The graphic  curve $\zeta: [0,1]\ni t \to (t,g(t))$ is a planar SEP, but it is not absolutely
continuous, in particular with this parametrization $\zeta$  is not Lipschitz.

\begin{theorem}\label{corIBDpath} Let $\ga$ be a self expanding path and let $x(\cdot)$ be a continuous parametrization  of $\ga$  defined in  a real interval $ T$. Then,
\begin{equation}\label{IBDCmon}
|x(t^{''})-x(t')| \leq |x(t)-x(t')| \text{\quad for all \quad }  t'\leq t^{''} \leq t \in T.
\end{equation}
Moreover where $\dot{x}(t)$exists, the property
\begin{equation}\label{inNco(ga(t))}
\dot{x}(t)\in N_{co(\ga(t))}(x(t))
\end{equation}
holds. If $ x(\cdot) $ is parameterized  by using the  curvilinear
abscissa $ s,$ the formula
\begin{equation}\label{MP}
 \frac{dw(co(\ga(s)))}{ds} \geq \frac{2}{\omega_n}\int_{\widehat{N(x(s))}} \langle
\theta,x'(s) \rangle d \sigma
\end{equation}
holds.
\end{theorem}
\begin{proof}
 \eqref{IBDCmon} follows from \eqref{IBDCmonlinearly}. In other words, for all $t'\in T$ the function
\begin{equation}\label{IBDCdef}
F(\cdot,x(t')): t \to |x(t)-x(t')|^2
\end{equation}
is a  not decreasing function in $t\geq  t' \in T$, with derivative greater than zero a.e., this implies 
\eqref{inNco(ga(t))}.
Inequality (\ref{MP})  it  is in 
(\cite{Manselli-Pucci}, theorem VII). It follows also from (\ref{firstvariationwidtheq}). 
\end{proof}
\begin{rem}\label{defselfexpanding}
Using \eqref{firstvariationwidtheq} it can be proved that actually equality holds a.e.  in (\ref{MP}).
 \end{rem}

\section{Quasi convex families}\label{quasiconvexfamilies}
Nested families of convex sets have been studied by De Finetti \cite{Defi} and Fenchel
\cite{fenc}.
\begin{definition}\label{defstratifications}
Let us call {\bf convex stratification},(see \cite{Defi}), a non
empty family $\mathfrak{F}$ of  convex bodies in $\rn$, linearly
strictly  ordered by inclusion ($\Om_{1} \subset \Om_{2} $,
$\Om_{1} \neq \Om_{2} $), with a maximum set and a minimum set, not identical.
\end{definition}
Let us call a parametrization of $\mathfrak{F}$ the inverse of a
strictly  increasing map of
 $\mathfrak{F}$ into a subset $W_{\mathfrak{F}}$ of a compact interval $T\subset \RR$.
\begin{definition}\label{defquasiconvexfamily} Let $\mathfrak{F}$ be a parameterized convex
stratification with a parametrization satisfying $W_{\mathfrak{F}}\equiv T$. The family
 $\mathfrak{F}$ will also be denoted $\{\Om_t\}_{t\in T}$. If for
 every
 $s \in T\setminus \{\max T \}$ the property:
$$\bigcap_{t>s}\Om_t=\Om_{s}$$
holds, then
as in \cite{fenc}, $\{\Om_t\}_{t\in T} $ will be called a {\bf quasi convex family}.
\end{definition}
In \cite{fenc} was noticed  that $\{\Om_t\}_{t\in T} $ is a quasi
convex family iff there exists  a lower semi continuous quasi convex
function, with $\{\Om_t\}_{t\in T} $ the family of its sub level
sets.

 Let  $\mathfrak{F}$,  $\mathfrak{G}$ be two
convex stratifications. Let us say that $\mathfrak{F}$ is contained in  $\mathfrak{G}$
if every element of $\mathfrak{F}$ is an element of $\mathfrak{G}$.

\begin{definition}\label{connectedstratifications}
A quasi convex family $\mathfrak{F}$ will be called connected if
$$\forall x \in \max\mathfrak{F} \setminus \relint \min\mathfrak{F}
 \quad \exists \, Q\in
\mathfrak{F} : x\in \partial_{rel} Q.$$
\end{definition}
Let $\mathfrak{F}$ be a connected  quasi convex family; in  $\mathfrak{F}$ let us consider
the usual Hausdorff distance between compact sets. A parametrization of $\mathfrak{F}$  will
 be called continuous if the map from the metric space $\mathfrak{F}$  to $T=W_{\mathfrak{F}}$
   is continuous.

  \begin{rem}
 Let us notice  that a quasi convex family may be not connected. As example let us
consider the family $\{\Om_{\tau}\}$ of the sub level sets of a
continuous quasi convex function $f$ with a flat zone; a flat zone
for $f$ is  an annulus with interior points
 bounded by two convex bodies, where $f$ is constant. For any interior point $x$ in the
 annulus does  not exist a level subset $\Om_{\tau}$ with $x \in \pa_{rel}\Om_{\tau}$.
  \end{rem}
In \cite{fenc} it is noticed that
\begin{equation}\label{fenchel1}
 \Omega_s = \overline{ \cup_{t<s} \Omega_t } \quad  \forall s \in T
 \setminus \{\mbox{minT}\}.
\end{equation}
is a necessary condition for  $ \mathfrak{F}$ to be a family of sublevel sets for a convex function.

\begin{lemma}\label{fenchel10}A quasi convex family $ \mathfrak{F} = \{ \Omega_t \}_{t \in T}
$ is connected iff \eqref{fenchel1} holds.
\end{lemma}
\begin{proof}
(\ref{fenchel1}) $ \Rightarrow   \mathfrak{F} $ is connected.

Let $ x \in \max\mathfrak{F} \setminus \relint\min
\mathfrak{F}. $ If  $  \{ t \in T: \relint \Omega_t \ni x
\} $ is empty, then  $ x \in \pa_{rel}\max \mathfrak{F}; $ if not,
let $t_2 := \inf \{ t \in T: \relint \Omega_t \ni x \}; $
then $ \Omega_{t_2} = \cap_{t>t_2} \Omega_t \ni x.$ If $ t_2 =
 \min T, $ then $ x \notin \relint
\min  \mathfrak{F} ,$ and $ x \in
 \pa_{rel} \min  \mathfrak{F} .$
If $ t_2 > \min T, $ then by (\ref{fenchel1}),
$$
\Omega_{t_2} = \overline{ \cup_{t<t_2} \Omega_t};
$$
as $ \Omega_t \not \, \ni x \: ( t < t_2 ), $ then $ x \notin
\cup_{t<t_2} \Omega_t  $ and thus $ x \in \pa_{rel} \Omega_{t_2}.$

 $\mathfrak{F} $ is connected $ \Rightarrow $ (\ref{fenchel1}).

 Assume, by contradiction, that there exists $ s_0 \in T \setminus
 \{\min T \}$ satisfying
 $$
 \overline{ \cup_{t<s_0} \Omega_t} \subset \Omega_{s_0} \quad
 \mbox{and} \quad
\overline{ \cup_{t<s_0} \Omega_t} \neq \Omega_{s_0} ;
 $$
 then
  $  \relint \Omega_{s_0} \setminus ( \overline{ \cup_{t<s_0}
  \Omega_t}))
   \neq \emptyset. $
Thus there exists $ x_0 \in  \relint\Omega_{s_0}, $ $
x_0 \notin \Omega_t ( t< s_0 ), $ $  x_0 \in  \relint 
\Omega_{t}$ $ ( t> s_0 ),  $ contradicting the hypothesis that
$\mathfrak{F} $ is connected.
\end{proof}

\begin{lemma}Let $  \mathfrak{F}=\{\Om_t\}_{t\in T}$ be  a  connected quasi convex family and
let $K$ be a convex body
satisfying
$$\min \mathfrak{F} \subseteq K \subseteq \max \mathfrak{F}.$$
Then,  there are $\Om_{t_1},\Om_{t_2}\in \mathfrak{F}$ satisfying
$\Om_t\supseteq  K\, \mbox{ iff } \, t\geq t_2, $ and $ \Om_t\subseteq K
$ iff $   t\leq t_1.$
\end{lemma}
\begin{proof} If $K=\min  \mathfrak{F}$ or $K=\max  \mathfrak{F}$ the lemma is obvious. Let us assume
 $\min \mathfrak{F} \neq K \neq \max \mathfrak{F}.$ Let $t_2=\inf\{t\in T: \Om_t \supseteq K\}$; then
$\Om_{t_2}=\bigcap_{t >t_2}\Om_t \supseteq K.$ Moreover $\Om_t
\supseteq K $ if and only if $ t \geq t_2$. Let $t_1=\sup\{t\in T:
\Om_t \subseteq K\}$ and $A=\bigcup_{t\in T}\{\Om_t\subseteq K\}$;
then $\Om_{t_1}\supseteq \cl A $  and $ \Om_t\subseteq K $ is not
possible if $ t> t_1. $ As $\mathfrak{F}$ is
connected, by previous lemma  $\Om_{t_1}=
\cl A$, thus $ \Om_t \subseteq K $ if and
only if $t\leq t_1$.
\end{proof}
\begin{theorem}\label{connectedimpliseregular}
Let $  \mathfrak{F}$ be  a connected quasi convex family, then  $
\mathfrak{F}, $ with the Hausdorff distance, is a connected complete metric
space; moreover   the mean width parametrization:
$w=w(K),\, K\in \mathfrak{F}$, is a continuous parametrization on the compact interval $[w(\min \mathfrak{F}),w(\max \mathfrak{F})]$.
On the other hand if $  \mathfrak{F}$ is   a   convex stratification and $w(\mathfrak{F})=[w(\min \mathfrak{F}),w(\max \mathfrak{F})]$,  then $  \mathfrak{F}$ is  a connected quasi convex family.\end{theorem}
\begin{proof}
The family $\mathfrak{K}$ of all compact  convex subsets  of $\max
\mathfrak{F},$ with the
 Hausdorff distance is a complete metric space by Blaschke's selection theorem
  (see e.g. \cite{Bonnfen}). Let $K\in \mathfrak{K}$,
  $\Om^{(l)} \in \mathfrak{F}$ such that
\begin{equation}\label{limOm_n}
\lims_{l\to \infty}\dist (\Om^{(l)},K) =0.
\end{equation}
Let us show that $K\in \mathfrak{F}$. Let
$\Om_{t_1}, \Om_{t_2}$ as in the previous lemma.  If
$\dist(\Om_{t_2},K)=0$ then $K=\Om_{t_2}\in \mathfrak{F}$; similarly
if $\dist(\Om_{t_1},K)=0$ then $K=\Om_{t_1}\in \mathfrak{F}$. From
\eqref{limOm_n}, the case  that $\dist(\Om_{t_2},K)>0 $ and $\dist(\Om_{t_1},K)>0$ cannot occur.

The mean width parametrization  $w(K), K\in \mathfrak{F}$ is a
strictly increasing  from the connected strictly linearly ordered
set $\mathfrak{F}$ to  $W_{\mathfrak{F}}=[w(\min \mathfrak{F}),w(\max \mathfrak{F})]$; since the Hausdorff distance on the elements of
 $\mathfrak{F}$ and the  mean width distance (see \eqref{distw<distH}
 and proposition \ref{problemII}) are equivalent, then $w: \mathfrak{F} \to W_{\mathfrak{F}}$
  is a one to one, strictly increasing function and its inverse is a continuous parametrization
  of  $\mathfrak{F}$. If  $\mathfrak{F}$ is a convex stratification and  $w(\mathfrak{F})=[w(\min \mathfrak{F}),w(\max \mathfrak{F})]$ then   $\mathfrak{F}$,
  with the parameter $w$, is a quasi
  convex family. The final part of the theorem follows by lemma \ref{fenchel10}. 
\end{proof}

\begin{theorem}[of completeness]\label{repcomp} Let
$\mathfrak{F}$ be a convex stratification. Then, there exists a
connected quasi convex family $\mathfrak{G}$ containing
$\mathfrak{F}$ so that $\min \mathfrak{G}=\min \mathfrak{F}$, $\max \mathfrak{G}=\max \mathfrak{F}$.
\end{theorem}
\begin{proof}
Let us parameterize the elements of the given family  $\mathfrak{F}$
by their mean width
 parameter $\tau=w(\Om_\tau)$, for $\Om_\tau \in \mathfrak{F}$. Let $\Sigma:=[w(\min \mathfrak{F}),w(\max \mathfrak{F})]$,
 then
$$\Om_{\tau_1} \subset \Om_{\tau_2}, \, \Om_{\tau_1} \neq \Om_{\tau_2}\text{\quad iff \quad }
 \tau_1 , \tau_2 \in w(\mathfrak{F}), \, \tau_1 < \tau_2.$$

If $\Sigma\setminus w(\mathfrak{F})=\emptyset$ the theorem is proved. If  $w(\mathfrak{F})$ is not closed, let $s \in cl(w(\mathfrak{F}))\setminus w(\mathfrak{F})$.  Let us add to $\mathfrak{F}$
 the convex body (that will be called $\Om_{s}$), obtained by limit of convex bodies  of
 $\mathfrak{F}$. This is well defined, since $\mathfrak{F}$ from the previous theorem
  is a subset  of the complete metric space of all compact
 convex subsets of $\max \mathfrak{F}$. Let us close $\mathfrak{F}$  according to this
  topology and let us call again $\mathfrak{F}$  the new completed family. The function $w$
  can be  extended in a continuous way to the augmented  family $\mathfrak{F}$. If $\Sigma\setminus w(\mathfrak{F})=\emptyset$ the theorem is proved. If not, 
 $w:\mathfrak{F} \to  \Sigma$ is a strictly increasing   continuous function and
  $\Sigma\setminus w(\mathfrak{F})$ is union of numerable relatively open intervals with  end points in
   $w(\mathfrak{F})$. Let
 $\tau\in \Sigma\setminus w(\mathfrak{F})$. Let $(\tau_1 , \tau_2)$ the maximal interval
   enclosed in $\Sigma\setminus w(\mathfrak{F})$ containing $\tau$. Then let us define for
   $\tau_1 < \la < \tau_2$ the interpolation between the convex sets
$\Om_{\tau_1},\Om_{\tau_2}$:
\begin{equation}\label{interpolation}
A_{\la}=\{ x\in \Om_{\tau_2}: \dist(x,\Om_{\tau_1})\leq \dist(\Om_{\tau_1},\Om_{\tau_2})
\frac{\la-\tau_1}{\tau_2-\tau_1}\}.
\end{equation}
The convex set $A_{\la}$ is the intersection between $\Om_{\tau_2}$ and the parallel
convex body to $\Om_{\tau_1}$, at  distance  ${(\la-\tau_1})/{(\tau_2-\tau_1)}$ from
$\Om_{\tau_1}$. For $\tau_1 < \tau < \tau_2$ let
$$\Om_\tau:=A_\la \text{\quad iff \quad } w(A_\la)=\tau,$$
and let us add these sets to the initial family. Let $\mathfrak{G}:=\{\Om_\tau\}_{\tau
 \in \Sigma}$ be the augmented family. $\mathfrak{G}$ is parameterized by its mean width parameter and 
$w(\mathfrak{G})=[w(\min \mathfrak{G}),w(\max \mathfrak{G})]$; then, by previous theorem, $\mathfrak{G}$ is a connected quasi convex family.
\end{proof}

\begin{definition}\label{defconvergg}Let $\mathfrak{K}$ be the space of all compact convex subsets  of $ \Om$ equipped with the Hausdorff distance.
Let $\mathfrak{G}^{(m)}=\{\Om_w^{(m)}\}_{w\in [w(\Om_0),w(\Om)]}$ a sequence of connected quasi convex
families parameterized by the mean width $w$, satisfying $ \min \mathfrak{G}^{(m)}=\Om_0$ and $ \max \mathfrak{G}^{(m)}=\Om$.
 Let us define 
$$\lims_{m\to \infty}\mathfrak{G}^{(m)}=\mathfrak{G}=\{\Om_w\}_{w\in [w(\Om_0),w(\Om)]}$$
if the continuous functions $w\to \Om_w^{(m)}$ uniformly converge to $w\to \Om_w $.
\end{definition}

\section{Steepest descent  curves  for quasi convex families}

\label{quasiconvexSDPcurves} Let $u$ be a smooth  function defined
in a convex body $\Om$.  Let $Du(x)\neq 0, \forall x\in \Om:  u(x) > \min u
$. A classical steepest descent curve of $u$ is  a rectifiable curve
$s\to x(s)$  solution to
$$ \dfrac{dx}{ds}= \dfrac{Du}{|Du|}(x(s)) $$
(some authors call them steepest descent curves with ascent parameter or steepest ascent
 curves). Classical steepest descent curves are the integral curves of a unit   field normal
 to the sub level sets  of the given function $u$. Here we are interested to convex
  sub levels, i.e. when $u$ is a quasi convex function. Let us consider the family of the
  sub level sets of $u$: $\Om_t=\{x\in \Om : u(x) \leq t\}$; let us notice that the family $\{\Om_t\}$ is a connected quasi convex family.
  
 Let us give now an extended definition 
  of a steepest descent curve related to a connected quasi convex family. 
\begin{definition}\label{defSEPclassical} Let $T$ be a closed real interval and let  $\{\Om_t\}_{t\in T}$
 be a connected quasi convex family. A continuous path  $t \to x(t)$ will be called a
   {\bf viable steepest descent curve} for $\{\Om_t\}_{t\in T}$ if
\begin{enumerate}
\item[(i)]$x(t) \in \pa_{rel}\Om_t $ $\quad  \forall t\in T\setminus\{\min T\}$;
\item[(ii)] $t\to x(t)$  is  a solution in $T$ of the  differential inclusion problem:
\end{enumerate}
\begin{equation}\label{defSDP}
\dot{x}(t) \in N_{\Om_t}(x(t)) \quad \text{a.e. in}  \quad T.
\end{equation}
$x(\max T)$ will be called the end point of the steepest
descent curve.
\end{definition}
Every self expanding path $\ga$ parameterized with arc length $s$ is a viable steepest descent curve for the family $\{co(\ga(s))\}, 0\leq s \leq ||\ga||$, see (\ref{inNco(ga(t))}).

For suitable quasi convex families is not possible  to get existence results of viable
 steepest descent curves, as  the following  example in $\RR^3$ shows.
\begin{ex}\label{exellipsoids} Let $E_t, t\in (1,2]$ be  the family of convex sets, defined by the rotations around the $x_3$ axis 
 of the following   plane sets:  union of the semicircles
$$(x_1-1)^2+x_3^2\leq \ (t-1)^2 ,\, x_1\geq 1  $$
and rectangles
$$|x_3|\leq t-1, \quad 0\leq x_1\leq 1.$$
Let $D_t, t\in [0,1]$ be the family of  circles in the plane $x_3=0$
$$ x_1^2+x_2^2 \leq t^2  .$$
\end{ex}
The  family $\{\Om_t\}_{t\in[0,2]}= \{D_t\}_{t\in
[0,1]}\bigcup\{E_t\}_{t\in (1,2]}$ is a connected quasi convex
family. Any viable  steepest descent curve $x(\cdot)$ of
$\{E_t\}_{t\in (1,2]}$ with end point $(\overline{x_1},\overline{x_2}, \pm 1)\in \partial E_2$
is a segment;  moreover if $\overline{x_1}^2+\overline{x_2}^2 < 1 $, then $x(\cdot)$ stops at time $t=1$ at a  point $x(1)\in \relint
D_1 $. Any continuous extension of $x(\cdot)$ to the interval $[0,1]$ as SEP
of $\{\Om_t\}_{t\in[0,2]}$ does not satisfy (i) in the definition
\ref{defSEPclassical} for $t\in (\tau, 1]$ with $\tau=|x(1)|^2$;
then it is not a  viable steepest descent curve.

Let us give a  definition  both extending     the viable steepest descent curves of the
definition \ref{defSEPclassical} and generalizing   the class of SEP of the previous section.
The aim of the following definition is to bind the natural order  structure of a
 quasi convex stratification with the natural order structure of an associated self
 expanding path.
\begin{definition}\label{defEC+} Let $\mathfrak{F}$ be a  convex stratification and
let $\ga$
 be a self expanding path enclosed in $\max \mathfrak{F}$. The couple $(\ga,\mathfrak{F})$
 will be  called an {\bf expanding  couple}  (EC) if:
\begin{enumerate}
\item[(i)] $\forall Q \in \mathfrak{F}, \ga \cap Q \neq \emptyset$,
\item[(ii)] $\ga \cap \pa_{rel} \max \mathfrak{F}\neq \emptyset$,
\item[(iii)] $\forall Q \in \mathfrak{F},\forall y \in Q, \forall  x \in \ga:
 x \not \in \relint Q ,    $

the properties
\begin{equation}\label{ineqdefEC+}
\forall x_1 \in \ga: x \prec x_1 \Rightarrow |x-y| \leq |x_1-y|
\end{equation}
hold.
\end{enumerate}
\end{definition}
\begin{rem}\label{EcinEC}Let $\mathfrak{F}$ be a convex stratification; let $(\ga, \mathfrak{G})$ be an EC
with $\mathfrak{F}\subset \mathfrak{G}$. Then $(\ga, \mathfrak{F})$ is  an EC too.
\end{rem}

\begin{theorem}\label{SACisEC}
If $x(\cdot)$ is a viable steepest descent
curve for a connected quasi convex family
$\mathfrak{G}=\{\Om_t\}_{t\in T}$ and  $t \to x(t)$ is absolutely continuous then, $(x(\cdot),\mathfrak{G})$
is
 an expanding  couple; however there exists a viable steepest descent
curve $\ga$ for  a connected quasi convex family $\mathfrak{G}=\{\Om_t\}_{t\in T}$, with $(\ga,\mathfrak{G})$ not expanding couple.

\end{theorem}
\begin{proof}
Let us observe that   since $x(\cdot)$ is absolutely continuous then
$t \to |x(t)-y|^2 $  is
 not decreasing  if and only if $$ \langle \dot{x}(t), x(t)-y\rangle\geq 0 ,
\quad \forall y \in \Om_{t'},\, \mbox{a.e.} t \geq t'.$$ The previous inequality is equivalent to the differential inclusion    \eqref{defSDP}  of definition \ref{defSEPclassical}. Thus $(x(\cdot),\mathfrak{G})$
is an expanding  couple.

To construct an example of viable steepest descent curve $\ga$ associated to $\mathfrak{G}$ such that $(\ga, \mathfrak{G})$ is not an expanding couple,
let us consider the Cantor function $g: [0,1]\to [0,1]$, see \cite[p. 83]{Halmos}. Let $\ga$ be 
the graph of $g$ in  the $x_1,x_2$ coordinate plane. The parametrization of $\ga: x(t)=( x_1(t)=t, x_2(t)=g(t)), t\in [0,1]$ is not absolutely continuous.
Let
$$\Om_t=co(\{(x_1,x_2): 0\leq x_1\leq t, g(t)\leq x_2\leq 1\}) \quad \mbox{for \quad } t \in [0,1].$$
As the Hausdorff distance between $\Om_{t_1},\Om_{t_2}$ is $|t_1-t_2|$, then $\{\Om_t\}_{t\in T}$ is connected.
 $\dot{x}(t)=(1,0)$ exists
a.e in $[0,1]$  and $\dot{x}(t)\in N_{\Om_t}(x(t)) $ since the halfplane $\{x_1\geq t\}$ supports $\Om_t$ at $x(t) $.  Of course 
$x(t)=(t,g(t)) \in \pa\Om_t$, so $\ga$ is a viable
steepest descent curve for $\{\Om_t\}_{t\in [0,1]}$.  Let us notice now that 
$$x(2/3)=(2/3,1/2)\not\in \relint\Om_{2/3},$$
and let us consider $x(1)=(1,1), y\equiv(2/3,1)\in \Om_{2/3}$.
Then $(\ga,\mathfrak{G})$ is not an EC since 
$$|x(2/3)-y|=1/2 > |x(1)-y|=1/3$$
holds. \end{proof}
 To construct an expanding couple with the related curve which is not a viable steepest
descent curve, let us consider the  family
 $\{\Om_t\}_{t\in[0,2]}= \{D_t\}_{t\in [0,1]}\bigcup\{E_t\}_{t\in (1,2]}$ of the example
  \ref{exellipsoids}.  Let $x(\cdot)$ be the continuous path with end point
   $\overline{x}=(\overline{x_1},\overline{x_2},1)\in \partial E_2\bigcap \{ x_3=1\}\bigcap\{0< x_1^2+x_2^2 < 1\} $ defined as follows:
\begin{equation}\label{x(t)par}
 x(t)=\left \{ \begin{array}{l}
(\overline{x_1},\overline{x_2},t) \quad {t\in (1,2]}\\
 (\overline{x_1},\overline{x_2},1)\quad {t\in (\tau,1]} \mbox{ with } \tau=\overline{x_1}^2+\overline{x_2}^2> 0\\
 \frac{t}{\tau}\cdot x(\tau) \quad t\in [0,\tau].
\end{array}
\right.
\end{equation}

It is not difficult to see that the constructed curve together with the family of example
\ref{exellipsoids}  is an EC but $x(\cdot)$ is not a viable steepest descent curve.

\begin{theorem}\label{Ec char} Let $\mathfrak{F}$ be a  convex
stratification and let $\ga$ be
 a self expanding path enclosed in $\max \mathfrak{F}$. Assume that $(\ga,\mathfrak{F})$
 satisfies
 (i) and (ii) of the  definition \ref{defEC+} . The following two facts
 are equivalent.
  \begin{enumerate}
\item[(i)]
The couple $(\ga,\mathfrak{F})$  is an expanding  couple (EC);
\item[(ii)] $\forall Q \in \mathfrak{F},$ $ \forall \, x'\in \ga \cap
\pa_{rel} Q, $ $ \forall \, \ga_1 \subset Q,  \ga_1 $ a SEP
with endpoint $ x'$, then the set $ \ga_2 := \ga_1 \cup (\ga \setminus
\ga_{x'}) $ (linearly ordered starting with the first point of $
\ga_1 $) is a SEP.
 \end{enumerate}
\end{theorem}

\begin{proof}(i)$ \Rightarrow $(ii)

Let $ x_1, x_2, x_3 \in \ga_2 $ with $  x_1 \prec x_2\prec x_3. $ If
 $ x_1, x_2, x_3 \in \ga_1 $ or  $ x_1, x_2, x_3 \in \ga \setminus
\ga_{x'},$ then  (\ref{IBDCmonlinearly}) holds. If $  x_1 \in \ga_1
\subset Q $ and $ x_2, x_3 \in \ga \setminus \ga_{x'},$ then
(\ref{IBDCmonlinearly}) follows from (\ref{ineqdefEC+}). If
 $ x_1, x_2 \in \ga_1 $ and $ x_3 \in \ga \setminus \ga_{x'},$
 then $ |x_1-x_2| \leq |x_1-x'|; $ as $ x_1 \in Q, $ by (\ref{ineqdefEC+})
$ |x_1-x'| \leq |x_1-x_3|; $ the last two inequalities imply
(\ref{IBDCmonlinearly}). So $\ga_2 $ is a SEP.

(ii)$ \Rightarrow $(i)

Let $ Q \in \mathfrak{F}, \; y\in Q, $ $\max( \ga \cap \pa_{rel} Q)
= x'. $ Let $ \ga_1 $ be the segment $ y x' $ and $ \ga_2 := \ga_1
\cup (\ga \setminus \ga_{x'}): $ as $ \ga_2 $ is a SEP then
(\ref{ineqdefEC+}) holds.
\end{proof}

\begin{theorem}\label{uovo}
Let $K_1\subset K_2$ be convex bodies in $\rn$ and let $\mathfrak{F}$ be a convex
stratification with minimum and maximum sets $K_1,K_2$ respectively. Let
 $(\ga,\mathfrak{F})$ be an expanding couple, then
\begin{enumerate}
\item[(i)] let $  c^{(1)}_n $ be the constant in theorem
\ref{selfsecrettifiable};
then
\begin{equation}\label{uovocolombo}
 (2 \, c^{(1)}_n)^{-1} \; ||\ga\setminus K_1||\leq \dist (K_1,K_2);
\end{equation}
\item[(ii)] there exists a constant $c$ depending on the diameter  of $K_2$ such that
the bound
\begin{equation}\label{uovocolombo2}
||\ga\setminus K_1||\leq c (w(K_2)-w(K_1))^{1/p}
\end{equation}
holds for $p=n$;
\item[(iii)] when $n>1, $  for any $p\geq 1,$ does not exist a constant $c, $ not depending on
$K_2,$ for which
 \eqref{uovocolombo2} holds.
\end{enumerate}
\end{theorem}
\begin{proof}
Let $x$ be the last point of $\ga$ and $x'$ be the projection of $x$ onto $K_1$, then
\begin{equation}\label{distprojection}
|x-x'| \leq \dist(K_1,K_2).
\end{equation}
The SEP $\ga$, by theorem \ref{Ec char}, can be extended  and made
it starting in $x'\in K_1$. Then
$$\ga \supseteq\ga \setminus K_1;$$
 by the monotonicity property of $\ga$
$$co(\ga)\subseteq B(x',|x-x'|).$$
Then  $ ||\ga \setminus K_1||\leq ||\ga || $ and by
\eqref{||ga||<=c(n)}, \eqref{distprojection}
$$
||\ga ||\leq \, c^{(1)}_n \, w(\mbox{co}(\ga ))
 \leq \, c^{(1)}_n \, w(B(x',|x-x'|)\leq \, 2 \, c^{(1)}_n \, \dist(K_1,K_2).
 $$
This proves (i). Inequality \eqref{uovocolombo2} follows immediately from (i) and
inequality \eqref{al=1/2}.

Let us observe now that, for any couple of nested convex bodies
$K_1,K_2,$ the segment joining the points  $x_i\in K_i,$
for $i=1,2$ such that $|x_2-x_1| =\dist(K_2,K_1)$ is
a special SEP which together with the trivial convex stratification
$\{K_i\}_{i=1,2}$ is an expanding couple. So in order to prove (iii)
it is enough to show that there exists a
 sequence of couples of
nested convex bodies $\emptyset \neq K_{1,\nu} \subset K_{2,\nu} $
of $\rn$ satisfying
$$\dfrac{\dist(K_{1,\nu} ,
 K_{2,\nu}) }{(w(K_{2,\nu})-w(K_{1,\nu} ))^{1/p} }\to \infty \quad
 \text{as} \quad \nu \to \infty.$$
The  example will be given in $\RR^2;$ however, it could be easily
adapted to $\rn$. Let $K_{1,\nu}$ be a family of segments of length
$\alpha_{\nu}/\nu$, where $\alpha_\nu $ is a suitable positive real
sequence to be determined in the sequel. Let us choose on the  axis
of $K_{1,\nu}$ a point $p_\nu $ of distance $1/\nu$ from
$K_{1,\nu}$. Let
$$K_{2,\nu}=co(K_{1,\nu}\cup \{p_\nu\}).$$
Then
$$\dist(K_{1,\nu},K_{2,\nu})=\dfrac{1}{\nu}, \quad w(K_{2,\nu})-w(K_{1,\nu})
=\frac{1}{\pi \nu}(\sqrt{4+\alpha_\nu^{2}}-\alpha_\nu).$$
The ratio
$$\dfrac{\dist(K_{1,\nu} ,
 K_{2,\nu}) }{(w(K_{2,\nu})-w(K_{1,\nu} ))^{1/p} } = \nu^{\frac{1}{p}-1}
\pi^{1/p} (\sqrt{4+\alpha_\nu^{2}}-\alpha_\nu)^{-1/p}$$ is unbounded
for $\nu \to \infty$  as $\alpha_\nu=\nu^q$, with $q >1$, $q/p
>1-1/p$.
\end{proof}

\begin{lemma}\label{lemmadiffinclusion} Let $(\ga,\mathfrak{F})$  be an expanding couple and
let $ Q \in \mathfrak{F}.  $ Then
%\begin{enumerate}
 $ \ga \cap  \pa_{rel}Q $ is at most one point.
%item[(ii)] Let $ x(t), \; t \in T $ be a  parametrization of $ \ga, $
% $ x(t) \not \in \relint Q; $ then $ dist(Q, \{x(t)\}) $ is a
%not decreasing function of $ t.$
%\item[(iii)] Let $ x(t), \, t \in T $ be a continuous parametrization of $ \ga. $
%If $ x(t_0) \in \ga \setminus \relint Q  $ and $ x'(t_0) $ exists,
%then
%$$
%x'(t_0) \in N_{co(Q\cup\{x(t_0)\})}(x(t_0)).
%$$
%\end{enumerate}
\end{lemma}
\begin{proof}
 The example \ref{exellipsoids} shows that  $ \ga \cap  \pa_{rel}Q
$ may be empty. Assume that there are $x, x_1 \in \ga \cap
\pa_{rel}Q,$ $ x \prec x_1.   $ Choose $ y = \frac{x + 2 x_1}{3} \in
Q; $ then (\ref{ineqdefEC+}) does not hold. 
%(ii)  follows easily from (\ref{ineqdefEC+}). From \eqref{ineqdefEC+}
%$$|x(t_0)-y|\leq |x(t)-y| \quad \text{for}\quad  t \geq  t_0, \forall y \in Q $$
%holds;  differentiating $|x(t)-y|^2$ at $t_0$ one gets
%$$\langle x'(t_0), x(t_0)-y \rangle \geq 0 \quad \forall y \in Q, $$
%and (iii) follows.
\end{proof}

The following is a theorem of completeness of EC.
\begin{theorem}\label{ECvsPEC} Let  $(\ga,\mathfrak{F})$ be an expanding  couple and let
$\Sigma=[w(\min \mathfrak{F}),w(\max \mathfrak{F})]$; then, 
 there exists   a connected quasi convex family (parameterized with respect to the mean width) $\mathfrak{G}=\{\Om_t\}_{t\in \Sigma},$
   containing the family $\mathfrak{F},$  and a continuous parametrization of $ \ga$: $\Sigma \ni t \to x(t) \in \ga$,
     with the properties:
\begin{enumerate}
\item[(i)] the couple $(\ga,\mathfrak{G})$ is an expanding couple with $\min \mathfrak{F}
=\min \mathfrak{G}$, $\max \mathfrak{F}
=\max \mathfrak{G}$;
\item[(ii)] $x(\cdot)$ is a continuous map from $\Sigma  \to \mathfrak{G}$;
\item[(iii)] for all $t\in \Sigma$ the point $x(t)\in \Om_t$; moreover $t'<t'', x(t')\neq x(t'')$ imply $x(t'')\not \in\Om_{t'}$;
\item[(iv)] $\forall \, t' \in [\min \Sigma, \max \Sigma),\forall y\in \Om_{t'}, $ the real function
$ t \to |x(t)-y|^2 $ is  not decreasing for $t\in (t',  \max
\Sigma]$.
\end{enumerate}
\end{theorem}
\begin{proof} Let us parameterize the elements  $Q$ of $\mathfrak{F}$ by  their mean
widths. The family
 $ \mathfrak{F}$ can be augmented to a convex stratification ${\mathfrak{F_1}}$, adding   the  sets
$$Q_s:=\lims_{\tau \to s, \tau \in w(\mathfrak{F})}Q_\tau, $$
if they are not present.
Then,  the couple $(\ga,{\mathfrak{F}_1})$ is still an expanding
 couple and
  $w(\mathfrak{F}_1)$ is closed. If ${\mathfrak{F}_1}$ it is not a connected quasi convex family,
 let us augment it in the following way. Let us consider the mean width function  $w$
  mapping $\mathfrak{F}_1$ to    a subset   $w(\mathfrak{F}_1)$ of $\Sigma$.  Let
 $\tau\in \Sigma \setminus w(\mathfrak{F}_1)$. Let $(\tau_1 , \tau_2)$ the maximal interval
  enclosed in $\Sigma\setminus w(\mathfrak{F}_1)$ containing $\tau$. Then let us consider
  the annulus between the convex sets
$Q_{\tau_1},Q_{\tau_2}\in \mathfrak{F}_1$. Let us assume    that
$$\ga_{\tau_1,\tau_2}:=\ga \cap \relint (Q_{\tau_2}\setminus  Q_{\tau_1})\neq \emptyset;$$
then $\ga_{\tau_1,\tau_2}$ is not a single point; let us complete
the stratification ${\mathfrak{F}_1}$ between
$Q_{\tau_1},Q_{\tau_2}$ in the following way: for  any $x \in
\ga_{\tau_1,\tau_2}$ let us add to ${\mathfrak{F}_1}$ the set
$co(Q_{\tau_1}\cup\ga_x)$. Let $\mathfrak{F}_2$ the augmented family
so obtained. Of course $\mathfrak{F}_2$ contains  $\mathfrak{F_1}$;
moreover by construction $(\ga,\mathfrak{F}_2)$ is an EC, $w(\mathfrak{F}_2)$ is closed  and
$w(\mathfrak{F}_2) \supset w(\mathfrak{F}_1)$. If $\Sigma \setminus w(\mathfrak{F}_2)$ is not empty, let $\tau \in
\Sigma\setminus w(\mathfrak{F}_2);$  let $(\tau_1 , \tau_2)$ the
maximal interval
  enclosed in $\Sigma\setminus w(\mathfrak{F}_2)$ containing $\tau$. The way as
  $ \mathfrak{F}_2 $ has been constructed implies that
  $\ga \cap \relint Q_{\tau_2}\setminus  \relint Q_{\tau_1}= \emptyset,$ but
$\ga \cap \pa_{rel} Q_{\tau_2}\cap\pa_{rel}
Q_{\tau_1}\neq\emptyset.$ At has been noticed above, $\ga \cap
\pa_{rel} Q_{\tau_2}\cap\pa_{rel}  Q_{\tau_1} $ consists in just one
point.
 Let us complete  $\mathfrak{F}_2$ in a
connected quasi convex family (interpolating the couples $Q_{\tau_1},Q_{\tau_2}$
as in \eqref{interpolation}). Let us call $\mathfrak{G}$ the 
augmented family; then
 $(\ga, \mathfrak{G})$ is an EC,  $w(\mathfrak{G})=\Sigma$ and $\mathfrak{G}$ is connected.

Now let us parameterize  $\ga$. Let $x\in \ga$. Let
$$t^{-}(x)=\sup \{w(Q): Q \in \mathfrak{G},  x\not \in Q \}, t^{+}(x)=\min\{w(Q):
 Q \in \mathfrak{G}, x \in Q\}.$$
Let $I(x)=[t^{-}(x), t^{+}(x)]. $ The set valued map $x \to I(x)$
from $\ga$ to the metric space of closed subintervals of $\Sigma$ is
strictly "monotone", to say
\begin{equation}\label{ord}
x_1\prec x_2 \Longrightarrow \max I(x_1) < \min I(x_2).
\end{equation}
This is a consequence of the following two properties:\\
(I) the subclass of sets of $Q\in \mathfrak{G}, x_1 \in Q$ is contained in those
containing $x_2$;\\
(II) does not exist an element $Q$ of $\mathfrak{G}$ , such that $\pa_{rel} Q$
contains the arc $\stackrel{\frown}{x_1\,x_2}$ of $\ga$.\\
Let $t\in \Sigma,$
then $t\in I(x)$ for a suitable $x\in \ga$.  This gives us a continuous parametrization
 $t\to x(t)$ of $\ga$. Then (ii) and the first sentence of (iii) are proved.
Now let $t'< t'', x(t')\neq x(t'');$ this implies $x(t') \prec x(t'')$; then 
\eqref{ord} gives us  $x(t'')\not \in\Om_{t'}$.
 Property (iv) follows from
 \eqref{ineqdefEC+}, since as observed above, $(\ga,\mathfrak{G})$ is an EC.
\end{proof}

\begin{definition} Let $(\ga, \mathfrak{G})$ be   an EC with $\mathfrak{G}=\{\Om_t\}_{t\in T}$ a parameterized connected quasi convex family. Let
\begin{equation}\label{parametrx(t)}
x(t)=\max\{x\in\ga: x\in \Om_t\}.
\end{equation}
The map $t\to (x(t),\Om_t)$ will be called a {\bf(joint) parametrization} of the EC $(\ga, \mathfrak{G})$.
\end{definition}

 Let $(\ga,\mathfrak{G})$ be an EC with $\mathfrak{G}$ connected. Let us point out that:
\begin{enumerate}
\item [(a)] as in the previous theorem and definition a continuous parametrization  $t$
for $\mathfrak{G}$ can be  used to give a  parametrization of
$\ga$. However $t \to x(t) \in \ga$  may have sets of
constancy. This occurs as example in the case a point $x \in \ga $
belongs simultaneously to the boundary of all sets $\Om_t$ for
$t_1\leq t \leq t_2$;
\item[(b)] a 1-1 parametrization in curvilinear abscissa of $\ga$ could be used to  make
 a parametrization of $\mathfrak{G}$: to a $Q \in  \mathfrak{G}$ it
 is associated $s$ such that $x(s) \in \pa_{rel} Q$;  but some sets
 of $\mathfrak{G}$ may be lost. See the curve $\ga$ defined by $\eqref{x(t)par}$ associated to  the family in the example \ref{exellipsoids} where
 a such  parametrization for $\mathfrak{G}$ would have jumps.
\end{enumerate}
\begin{lemma}
 Let  $(\ga, \mathfrak{G})$ be an EC, with $\mathfrak{G}$  a connected quasi convex family and if
 \begin{equation}\label{dimQ}
 \mbox{ for every $Q \in \mathfrak{G}$, $Q \neq \min \mathfrak{G}$, the
 dimension of $\mathit{Aff}(Q)$ is constant} 
 \end{equation}
then there exists a joint parametrization $t\to (x(t),\Om_t)$ of  $(\ga, \mathfrak{G})$ such that $t\to x(t)$
is a viable steepest descent curve for $\{\Om_t\}$.
\end{lemma}
\begin{proof}
 Let us assume with no loss of generality that the dimension of
 elements Q is $n$. Let us choose the curvilinear abscissa of $\ga$ as the parameter $t$.
 Let 
 $$\Om_t:= \min \{Q\in \mathfrak{G}: x(t)\in Q, \quad 0\leq t \leq ||\ga|| \} .$$
 The fact that $\mathfrak{G}$ is connected and \eqref{dimQ} imply that $x(t)\in \pa \Om_t$. 
 The proof of (ii) of the definition \ref{defSEPclassical} can be done as in the proof of theorem \ref{SACisEC}.
\end{proof}

\begin{rem}\label{x(t0)costant} Let $(x(t),\Om_t)_{t\in T}$ be a joint parametrization of an EC with $\{\Om_t\}_{t\in T}$ a  connected quasi convex family and let 
$x(t_0)\in \text{relint\ } \Om_{t_0}$. 
Then there exists an interval, not reduced to a point, containing $t_0$ where $x(t)=x(t_0)$. 
\end{rem}

The following problem is addressed and solved in what follows.\\
Given a connected quasi convex family $\mathfrak{G}=\{\Om_t\}_{t\in
\Sigma}$ and a point $\overline{x} \in \pa \max \mathfrak{G};  $
does it  exist a curve $\ga$  with end point $\overline{x}$ so that
the couple $(\ga,\mathfrak{G})$ is a EC ? Is it $\ga$ unique ?

\begin{definition}
Let $(\ga^{(m)},\mathfrak{G}^{(m)})$ be  a sequence of EC with $\mathfrak{G}^{(m)}$ connected quasi convex families and let $\min \mathfrak{G}^{(m)}=\Om_0$,
 $\max \mathfrak{G}^{(m)} = \Om$. Let    $s^{(m)}\in [0, l^{(m)}]$ be 
the arc length of $\ga^{(m)}$, and let  us choose for all of them the same parameter $t\in [0,1], t=s^{(m)}/l^{(m)}$.
 Let us define 
$$\lims_{m\to \infty}(\ga^{(m)},\mathfrak{G}^{(m)})=(\ga,\mathfrak{G})$$
if the sequence $[0,1]\ni t \to x^{(m)}(t)$ uniformly converges to $t\to x(t)$ of $\ga$ and $\mathfrak{G}^{(m)}$ converges to $\mathfrak{G}$ according to definition \ref{defconvergg}.
\end{definition}
\begin{theorem}\label{lemmacompactEC}
Let $(\ga^{(m)},\mathfrak{G}^{(m)})$ be a sequence of EC  with $\mathfrak{G}^{(m)} $ connected quasi convex families and let $\min \mathfrak{G}^{(m)}=\Om_0$,
$\max \mathfrak{G}^{(m)} = \Om$.
  Then there exists a subsequence which converges to an EC $(\ga,\mathfrak{G})$.
\end{theorem}
\begin{proof} Let us parameterize all
  $\mathfrak{G}^{(m)}$ by the mean width $w \in W= [w(\Om_0),w(\Om)]$,
  i.e. $\mathfrak{G}^{(m)}=\{\Om_{w}^{(m)}\}_{w\in W}. \; $
Then, each $\mathfrak{G}^{(m)} $ is represented as a continuous
function defined in $ W $ and valued into the metric space $
\mathfrak{K} $ of all compact subsets of $ \Om $ with Hausdorff
distance.
  From proposition \ref{problemII}
   these functions  are  equicontinuous; from Ascoli-Arzel\'a
   theorem \cite[p. 234]{Kelley}  there exists a   converging subsequence
   to a continuous function, corresponding to a family
   $\mathfrak{G}=\{\Om_{w}\}_{w\in W}$. Let us assume for simplicity  that
    the converging subsequence is the initial sequence.
Let us consider now the sequence of the corresponding curves
$\ga^{(m)}$. Let $l^m=||\ga^m||$. Theorem  \ref{selfsecrettifiable} implies 
$$0< l^m< c^{(1)}_n w(\Om).$$
Let us reparameterize  $\ga^{(m)}$ with  their arc
length $s^{(m)}\in [0, l^{(m)}]$, and let us choose for all of them
the same parameter $t\in [0,1], t=s^{(m)}/l^{(m)}$. From theorem
\ref{selfsecrettifiable} they are equilipschitz, then there exists a
subsequence converging to a self expanding path  $\ga=\{x(t), t\in
[0,1]\}$. It remains to prove that   $(\ga,\mathfrak{G})$ is a EC.

 (i) and (ii) of definition \ref{defEC+} hold with a limit argument.

 To prove (iii),
let $Q\in \mathfrak{G}$, $x\prec x' \prec x_1 \in \ga$,
$ x\not \in \text{relint\ }Q$, $x'=x(t'),\, x_1=x(t_1)$. The previous argument implies that
 $Q= \lims_{m\to \infty} \Om_{w}^{(m)}$, with $w=w(Q)$,
  $x=x(t)=\lims_{m\to \infty} x^{(m)}(t),\, x'=x(t')=\lims_{m\to \infty} x^{(m)}(t'), x_1=x(t_1)=\lims_{m\to \infty} x^{(m)}(t_1).$
   Since  $Q$ is compact, then eventually  $x^{(m)}(t')\not \in \Om_{w}^{(m)}$.
Since $(x^{(m)}(\cdot),\mathfrak{G}^{(m)})$ are  EC, they satisfy
definition \ref{defEC+}.
It follows that for $m$ large enough
$$\forall y \in \Om_{w}^{(m)} \Longrightarrow  |x^{(m)}(t')-y| \leq |x^{(m)}(t_1)-y|.$$
Then, for every
$ y \in Q \mbox{\quad  it holds \quad } |x'-y| \leq |x_1-y|.$ Thus, as $x'$ tends to $x$, \eqref{ineqdefEC+} holds.
\end{proof}

The following is an existence   result for expanding couples.
\begin{theorem}\label{coruniqcondEC} Let  $ \mathfrak{G}=\{\Om_t\}_{t\in T}$ be a connected
 quasi convex family in $\rn$ with $t=w(\Om_t)$, $T=[w(\min \mathfrak{G}),w(\max \mathfrak{G})]$ then
\begin{enumerate}
\item[(i)] for every $ \overline{x}\in \pa \max \mathfrak{G}$ there exists a self expanding path
 $\ga$, so that $(\ga,\mathfrak{G})$ is an EC with $\ga $ having end point $ \overline{x}$;
\item[(ii)] $\ga$ is rectifiable, with length  bounded by the mean width of $\mathfrak{G}$ times
a constant depending only on $n$;
\item[(iii)]  $\ga$  can be uniformly approximated with  piecewise linear
self expanding  curves
 $x^{(m)}(\cdot)$, related to  connected
 quasi convex families  $\mathfrak{G}^{(m)}$ (uniformly converging to
 $\mathfrak{G}$) and  $(x^{(m)}(\cdot),\mathfrak{G}^{(m)})$ are joint parameterized EC.
\end{enumerate}
\end{theorem}
\begin{proof} Let  $\overline{w}= w(\max
\mathfrak{G}), w_0=w(\min \mathfrak{G})$.

 Let
$D$ be a countable dense subset of $[w_0, \overline{w}]$, and let
 $w_0=w_1^{(m)} < \cdots < w_m^{(m)}=\overline{w}$ elements of $D$,
  satisfying
   $\{w_1^{(m)}, \ldots , w_m^{(m)}\}\subset \{w_1^{(m+1)}, \ldots , w_{m+1}^{(m+1)}\}$.
Let $\{\Om_{w_1^{(m)}}, \ldots , \Om_{w_m^{(m)}}\}$ be a finite
convex stratification (extracted from $\mathfrak{G}$) and let
$\{O^{(m)}_t\}_{t \in [w_0,\overline{w}]}$ the related connected
family as defined in the proof of theorem \ref{repcomp}. Let $t \to
x^{(m)}(t)$ the piecewise linear curve ending at $\overline{x}$,
union of
 segments joining the $m$  points   $p^{(1)}, \ldots, p^{(m)}, $  where
  $p^{(m)}=\overline{x}$,\quad $p^{(j-1)}$ is the projection of $p^{(j)}$ onto
  $\Om_{w_{j-1}^{(m)}}, j=2, \ldots ,m $.  Let
\begin{equation}\label{p(t)}
x^{(m)}(t)=\frac{(t-w_{j-1}^{(m)})p^{(j)}+(w_{j}^{(m)}-t)p^{(j-1)}}
{w_{j}^{(m)}-w_{j-1}^{(m)}}, \quad w_{j-1}^{(m)}\leq t \leq
w_j^{(m)},
\end{equation}
the corresponding  point between $p^{(j-1)}$ and $ p^{(j)}$.
If  $p^{(j-1)} \neq p^{(j)}$, the unit segment direction  $ d_j$ of $p^{(j-1)} p^{(j)}$,  is orthogonal at each point
 $x^{(m)}(t)$ of $p^{(j-1)} p^{(j)}$ to the boundary of convex set $O^{(m)}_t$ which is
 an interpolation between $\Om_{w_{j-1}^{(m)}}$ and $\Om_{w_{j}^{(m)}}$ as
 defined in (\ref{interpolation}).
%to the normal cone of the simple cap body  $(O^{(m)}_{w_{j-1}^{(m)}})^{p(t)}$;
Therefore
$$\dfrac{d x^{(m)}}{dt}(t)\in N_{O_t^{(m)}}(x^{(m)}(t)) \quad \text{for}
 \quad w_{j-1}^{(m)}\leq t \leq
w_j^{(m)} .$$
Thus for a small  $\vare >0$, $p^{(j-1)} p^{(j)}\cap O^{(m)}_{w_{j-1}^{(m)}+\vare} $ is the trajectory of a viable  steepest descent
 curve of the quasi convex family 
  $$\{O^{(m)}_{t}\}_{t\in [w_{j-1}^{(m)}+\vare ,w_j^{(m)}-\vare]}.$$
From theorem \ref{SACisEC}, 
$$t\to (x^{(m)}(t),O^{(m)}_{t})_{t\in [w_{j-1}^{(m)},w_j^{(m)}]}$$
is a joint parametrization of $(x^{(m)}(\cdot),\{O^{(m)}_{t}\}_{t\in [w_{j-1}^{(m)},w_j^{(m)}]}$; of course 
  $T\ni t \to x^{(m)}(t)$ is a  SEP,  and
  $(x^m(\cdot), \mathfrak{G}^{(m)})$ with $\mathfrak{G}^{(m)}=\{O^{(m)}_{t\in T}\}$ are
  expanding couples. By construction the sequence of families
    $\mathfrak{G}^{(m)}$ converges uniformly in the Hausdorff distance to
    $\mathfrak{G}$.  From  theorem \ref{lemmacompactEC}  there exists a
  subsequence of $(x^m(\cdot), \mathfrak{G}^{(m)})$ converging to an expanding couple 
  $(\ga,\mathfrak{G})$. Thus (i) and  (iii) are proved and (ii) follows from \eqref{||ga||<=c(n)}.
\end{proof}
From  theorem \ref{repcomp}, previous theorem and remark \ref{EcinEC} it follows 
\begin{cor}Let $\mathfrak{F}$ be a convex stratification and let $\overline{x}\in \pa \max \mathfrak{F}$. Then, there exists a SEP $\ga$ with end point $\overline{x}$ so that $(\ga, \mathfrak{F)}$ is an EC.
\end{cor}

The following lemma and theorem show that given an EC $(\ga,\mathfrak{G})$ with  $\mathfrak{G}$ a connected quasi convex family, a joint parametrization for $(\ga,\mathfrak{G})$ that gives an absolutely continuous 
parametrization (indeed lipschitz) for  $\ga$ can be found.
\begin{lemma}\label{s(w)} Let $s:[w_0,\underline{w}]\to \R $ be a continuous not decreasing function and let $\eta$ be the plane 
curve ${(w,s(w)), w_0\leq w \leq \underline{w}}$. 
Then there exists a continuous, strictly increasing function $\tau: [w_0,\underline{w}]\to [0, ||\eta||]$,  with inverse $w(\cdot)$, so that $s(w(\cdot)): [0, ||\eta ||] \to \R$ is a lipschitz function.
\end{lemma}
\begin{proof}
Let 
$$\tau (w):=||\eta \cap([w_0,w]\times \R)||.$$
The function $w\to \tau (w)$ is a continuous, strictly increasing function (see e.g. \cite[theorem 8.4]{Saks}) with inverse
$$w:[0,|| \eta||] \to [w_0,\underline{w}].$$
Let $0\leq \tau_1 < \tau_2 \leq || \eta|| $, then:
$$|s(w(\tau_2))-s(w(\tau_1))|\leq ||\eta\cap([w(\tau_1),w(\tau_2)]\times\R)||=\tau_2-\tau_1.$$
This concludes  the proof. \ \end{proof}
\begin{theorem}\label{abscont}
Let $(\ga,\mathfrak{G})$ be an expanding couple with $\mathfrak{G}$ a connected quasi convex family. Then there exists 
a joint parametrization $(z(\tau), \Om_\tau )_{\tau \in T}$ of $(\ga,\mathfrak{G})$ so that $t\to z(t)$ is 
lipschitz.
\end{theorem}
\begin{proof}
Let us start from  the results of  theorem \ref{ECvsPEC}. 
If the mean width parameter $w$ of elements of the family $\mathfrak{G}=\{O_w\}_{w\in[w(\min \mathfrak{G}),w(\max \mathfrak{G})]}$ is used to parameterize
$\ga$ as in \eqref{parametrx(t)}, a continuous monotone parametrization  $x(w), w\in [w(\min \mathfrak{G}),w(\max \mathfrak{G})]$ of $\ga$ is obtained.
From theorem \ref{selfsecrettifiable} the curve $\gamma$ is rectifiable and can be represented as function of its arc length: $y(s), 0\leq s \leq ||\ga||$. 

Let  $s:[w(\min \mathfrak{G}),w(\max \mathfrak{G})] \to \R$ the map defined as
$$w \to s(w)=||\ga\cap O_w||=||\ga_{x(w)}||.$$
$s(\cdot)$ is  a continuous  not decreasing  function and $y(s(w))=x(w)$. From previous lemma there exists a continuous change of variable  $w=w(\tau)$ so that
$\tau \to s(w(\tau))$  is lipschitz continuous.
The connected quasi convex family  $\mathfrak{G}=\{O_w\}$ will be changed in $\Om_{\tau}=O_{w(\tau)}$ with
$\tau \in [0, || graph(s)|| ]$;  the associated  curve $\ga$ has the parametrization $z(\tau)= y(s(w(\tau)))$;
since   $s\to y(s)$ is lipschitz continuous,
 $z(\cdot):\tau \to y(s(w(\tau)))$, composition of two lipschitz continuous maps, is lipschitz continuous. 
\end{proof}

 \begin{ex}
Let  $[0,1] \ni t \to g(t)$ be the Cantor  function. Let $\mathfrak{F}=\{\Om_t\},$ with
$$
\Om_t =\{x \in \RR^2: |x| \leq \frac{g(t)+t}{2}, 0\leq t \leq 1\}
$$
 a connected quasi convex family of concentric circles, and let $x(\tau)=
 \frac{g(\tau)+\tau}{2}\overline{x}$ ($\overline{x}\in \partial \Omega_1$) be the
 parametrization of the radius from the origin to $\overline{x}$. It is a viable steepest descent
 curve for $\mathfrak{F}$, but its parametrization is not absolutely continuous.
 \end{ex}
However, let us notice that, 
using  in place of $t$ the mean width parameter $w$ of the family $\{\Om_t\}$, i.e. $w(\Om_t)=g(t)+t$ and the parametrization of $\ga$ given by \eqref{parametrx(t)},  the absolutely continuous property for the parametrization  $\ga$   is restored. 
Next   proposition shows that, given a connected   quasi convex family $\{\Om_t\}_{t\in T}$,
viable  steepest  descent  curves are  continuously depending on their end point, if they are absolutely continuous.

\begin{proposition}\label{uniqcondepSDP}
Let  $x(\cdot), y(\cdot)$ be two absolutely continuous viable steepest descent
 curves of   a connected quasi convex family $\{\Om_t\}_{t\in T }$ with end  points
 $\overline{x}, \overline{y}$ respectively. Then
\begin{equation}\label{SDPineq}
|x(t)-y(t)| \leq |\overline{x}-\overline{y}| \quad \forall
\, t \in T,
\end{equation}
and there is at most   one  absolutely continuous viable steepest
descent curve with given end point.
\end{proposition}
\begin{proof} Let $t$ be a value for which
$\dot{x}(t)\in N_{\Om_{t}}(x(t)), \, \dot{y}(t)\in
N_{{\Om_{t}}}(y(t)).$ Since  $x(t), y(t) \in \pa_{rel} \Om_{t}\,
$, the previous inclusions mean that
\begin{equation}\label{xdot(x-y)}
 \langle \dot{x}(t), x(t)-y(t)\rangle\geq 0 , \quad
 \langle \dot{y}(t), y(t)-x(t)\rangle \geq 0.
\end{equation}
Adding the two previous inequalities,   a.e. 
$$\frac{1}{2}\frac{d}{dt}|x(t)-y(t)|^2 \geq 0 $$
holds. The absolute continuity assumption implies that the distance between $x(t)$ and $y(t)$ is  not decreasing with respect
 to the level value $t$; this proves \eqref{SDPineq}.
\end{proof}
Next theorem gives us a continuous dependence for $EC$.
 
\begin{theorem}\label{uniqcondEC} Let $\mathfrak{G}=\{\Om_t\}_{t\in T}$ be  a  connected
   quasi convex family and
let $(x_1(\cdot),\mathfrak{G})$,$(x_2(\cdot),\mathfrak{G})$  be two  expanding  couples with joint continuous parametrizations; let
 $\overline{x_1},\, \overline{x_2}$  be the curves' end   points. 
  Then for all $t\in T$
\begin{equation}\label{eqcineq}
|x_1(t)-x_2(t)| \leq |\overline{x_1}-\overline{x_2}|
\end{equation}
and for any given end point there exists at most   one EC.
\end{theorem}
\begin{proof}
 Step 1.  Let $s_i(\cdot):[w_0, \overline{w}]\to \R$ be  continuous not decreasing functions ($i=1,2$).
Then, there exists a real continuous, strictly increasing function $\tau: [w_0,\overline{w}]\to \tau( [w_0,\overline{w}]):=T_1$,  with inverse map 
$w(\cdot)$, so that $s_i(w(\cdot)): T_1 \to \R$ is a lipschitz function ($i=1,2$).

The proof of step 1 is similar to the proof of lemma \ref{s(w)} where
$\eta $ is the three dimensional curve $[w_0,\overline{w}]\ni w \to (w,s_1(w),s_2(w))$ and 
$$\tau(w):=||\eta \cap([w_0,w]\times \R^2)||.$$
Step 2. Let $(\ga_i,\mathfrak{G})$ be two expanding couples ($i=1,2$). Then, there exists 
a joint parametrization $(z_i(\tau), \Om_\tau )_{\tau \in T_1}$ of $(\ga_i,\mathfrak{G})$ ($i=1,2$) so that $z_i(\cdot)$ is 
lipschitz   
continuous in $T_1$ ($i=1,2$).
The proof of step 2 is similar to the proof of theorem \ref{abscont} where 
$$w \to s_i(w)=||\ga_i\cap O_w||, \quad i=1,2$$
and $w(\cdot)$ is introduced in  step 1. 

 The conclusion of the proof of the theorem is similar to that of the proposition \ref{uniqcondepSDP}.  Let us use the joint parametrizations 
 $(z_1(\tau),\Om_\tau)),(z_2(\tau),\Om_\tau))$    
 introduced in the previous step. Let us recall that $\tau \to z_1(\tau), \tau \to z_2(\tau)$ are absolutely continuous in $T_1$. 
 If $z_1(\tau),z_2(\tau) \in \pa_{rel} \Om_\tau$ then \eqref{xdot(x-y)} holds for  such $\tau$. In case $z_1(\tau_0)\in \text{relint\ }\Om_{\tau_0}$ then,
 from
 remark \ref{x(t0)costant}, the  right or the left derivative of $z_1(\cdot)$ at $\tau_0$ is zero. Similarly is for $z_2(\cdot)$. 
 Then, there exists a.e. $\dot z_1(\tau), \dot z_2(\tau)$ and \eqref{xdot(x-y)} holds a.e. 
 Thus a.e. 
 $$\frac{1}{2}\frac{d}{d\tau}|z_1(\tau)-z_2(\tau)|^2 \geq 0. $$
 Therefore, by the absolute continuity property,
   $$|z_1(\tau)-z_2(\tau)| \leq |\overline{x_1}-\overline{x_2}|.$$
 Then  \eqref{eqcineq} holds for every joint parametrization of $(\ga_1,\mathfrak{G}),(\ga_2,\mathfrak{G})$. 
\end{proof}

 \begin{cor}Let $(\ga_i, \mathfrak{G})$ be two EC with $\mathfrak{G}$ a connected quasi convex family ($i=1,2$).
 Let $x_i=\ga_i\cap \max  \mathfrak{G} (i=1,2)$.
Then
 $$\dist (\ga_1,\ga_2)\leq |x_1-x_2|.$$  
 \end{cor}
 \begin{proof}
  By definition of EC,
  $$|x(t_1)-y(\tau)| \leq |x(t_1)-y(t_1)| , \quad \mbox{for} \quad \tau < t_1 ,$$
  since $y(\tau) \in \Om_\tau \subset \Om_{\tau_1}$. 
 \end{proof}

\section{Appendix}

\begin{lemma}\label{lemmappendix0}
Let $K_{v,\delta}$ be a circular $n$-dimensional cone of axis $v$, amplitude $\delta$
 greater than zero and less than $\pi/2$. Then
$$\int_{\widehat{K_{v,\delta}}}\langle \theta, v \rangle \, d\sigma(\theta) =
 \frac{\om_{n-1}}{n-1}\sin^{n-1}\delta.$$
\end{lemma}
\begin{proof}See e.g. \cite[pp. 223-224]{Manselli-Pucci}.
\end{proof}
 
\begin{lemma}\label{lemmappendix1}
Let $K_{v,\alpha/2}$ be a circular $n$-dimensional cone of axis $v$, amplitude
 $\frac{\alpha}{2}$ greater than zero and less than $\pi/2$. Let $u$ a unit vector
  in $K^*_{v,\alpha/2}$, then
$$\int_{\widehat{K_{v,\alpha/2}}}\langle \theta, u \rangle
\, d\sigma(\theta) \geq \frac{\om_{n-1}}{n-1}\sin^{n}(\alpha/4).$$
\end{lemma}
\begin{proof}
Let $v=e_n$. Let $\phi=\langle \theta, e_n \rangle, \psi=\langle  u,e_n \rangle.$
 For $\theta$ in the sector  $\widehat{K_{v,\alpha/2}}$, $0\leq \phi \leq \alpha/2;$
 and $0\leq \psi \leq \frac{\pi}{2}-\alpha/2$ since $u\in K^*_{e_n,\alpha/2}$.
  So $0\leq \phi+\psi \leq \pi/2$. Moreover the  distance on the sphere $S^{n-1}$
  between $\theta$ and $u$ is less  than the sum of the  distances between $\theta$
  and $v$, $v$ and $u$, i.e
$$\langle \theta, u \rangle \geq \cos (\phi+\psi ).$$
Let us consider the smallest sector $\widehat{K_{v,\alpha/4}}$;  the inequality
$$\int_{\widehat{K_{v,\alpha/2}}}\langle \theta, u \rangle \,
 d\sigma(\theta) \geq\int_{\widehat{K_{v,\alpha/4}}} \cos (\phi+\psi )\, d\sigma(\theta)$$
holds.
For $\theta \in \widehat{K_{v,\alpha/4}}$, the angle $0\leq \phi\leq \alpha/4$,
therefore $0\leq \phi+\psi\leq  \frac{\pi}{2}-\alpha/4$. Thus
$$\cos (\phi+\psi) \geq \cos(\pi/2-\alpha/4)=\sin (\alpha/4).$$
Then
$$\int_{\widehat{K_{e_n,\alpha/4}}} \cos (\phi+\psi )\, d\sigma(\theta)\geq
 \sin (\alpha/4)\int_{\widehat{K_{e_n,\alpha/4}}}\cos \phi \, d\sigma(\theta).$$
From previous lemma the proof is obtained.
\end{proof}

\end{document}